\documentclass[]{amsart}

\usepackage{amsmath, amsthm,amssymb}
\usepackage{lipsum}
\usepackage{graphicx}
\usepackage{epstopdf}
\usepackage{algorithm}
\usepackage{algorithmic}
\usepackage[inline]{enumitem}
\usepackage{tasks}
\usepackage[textsize=footnotesize]{todonotes}
\usepackage{comment}

\usepackage[]{url}
\usepackage[]{hyperref}
\hypersetup{
    colorlinks=true,
    linkcolor=blue,
    filecolor=magenta,      
    urlcolor=cyan,
}

\usepackage{cleveref}

\ifpdf
  \DeclareGraphicsExtensions{.eps,.pdf,.png,.jpg}
\else
  \DeclareGraphicsExtensions{.eps}
\fi

\newtheorem{theorem}{Theorem}[section]
\newtheorem{assumption}{Assumption}[section]
\newtheorem{corollary}[theorem]{Corollary}
\newtheorem{lemma}[theorem]{Lemma}
\newtheorem{proposition}[theorem]{Proposition}
\theoremstyle{definition}

\theoremstyle{remark}
\newtheorem{remark}{Remark}[section]

\newcommand{\highlight}{}
\newcommand{\mymark}{}   
\newcommand{\ves}{{\rm vec} } 
\newcommand{\vecs}{{\rm vech} }

\newcommand{\TheTitle}{Continuous-Time Robust Dynamic Programming} 
\newcommand{\TheAuthors}{T. Bian and Z. P. Jiang}

\title{\TheTitle}
\thanks{The opinions expressed in this paper are those of the authors and do not necessarily reflect the views and policies of Bank of America Merrill Lynch.
}

\author{Tao Bian}
\address[T. Bian]{Bank of America Merrill Lynch, One Bryant Park, New York, NY 10036}
\email[T. Bian]{tbian@nyu.edu}
\author{Zhong-Ping Jiang}
\address[Z. P. Jiang]{Control and Networks Lab, Department of Electrical and Computer Engineering, Tandon School of Engineering, New York University, 5 Metrotech Center, Brooklyn, NY 11201}
\email[Z. P. Jiang]{zjiang@nyu.edu}

\usepackage{amsopn}

\ifpdf
\hypersetup{
  pdftitle={\TheTitle},
  pdfauthor={\TheAuthors}
}
\fi

\begin{document}

\maketitle

\begin{abstract}
This paper presents a new theory, known as robust dynamic programming, for a class of continuous-time dynamical systems.
Different from traditional dynamic programming (DP) methods, this new theory serves as a fundamental tool to analyze the robustness of DP algorithms, and in particular, to develop novel adaptive optimal control and reinforcement learning methods. 
In order to demonstrate the potential of this new framework, four illustrative applications in the fields of stochastic optimal control and adaptive DP are presented.
Three numerical examples arising from both finance and engineering industries are also given, along with several possible extensions of the proposed framework.
\end{abstract}



\tableofcontents

\section{Introduction}
In 1952, Bellman proposed the original idea of dynamic programming (DP) \cite{Bellman1952} to solve a class of optimization problems subject to a controlled process that is usually described by a Markov decision process (MDP), a difference equation, or a differential equation.
Over the past several decades, DP and its extensions \cite{Puterman2005, Bertsekas2005, Bertsekas2007, Bertsekas2013} have attracted a significant amount of attention, because of the vital role they have played in several popular fields including reinforcement learning (RL) \cite{Sutton1998,Littman2015,Barto2017}, finance \cite{Merton1971, Tsitsiklis1999a, Pham2009, Garleanu2016}, and biological control \cite{Kleinman1970, Todorov2005}, to name a few.
Depending on the form (discrete-time vs. continuous-time) used to describe the dynamical system in question, DP problems can be solved by finding the solution to either the Bellman equation or the Hamilton-Jacobi-Bellman (HJB) equation.
However, due to the complex nature of these equations, the optimal solution cannot be obtained analytically in most cases, and numerous  methods including policy iteration (PI) \cite{Howard1960, Kleinman1968, Beard1995, Bian2014} and value iteration (VI) \cite{Bellman1957a, Bertsekas2017a, Bian2016b, Bian2016f} have been developed to approximate the solutions of these equations. 
Unfortunately, these algorithms suffer from serious usage limitations, due to the limited information available and the presence of various types of disturbance in practical problems.
Nevertheless, from a control theory  point of view, we identify two perspectives to address these issues.
The first one, which we refer to as the ``adaptive control perspective", aims at learning the unknown components in DP algorithms directly from available online/offline data. 
Based on the problem formulation, such unknown component can be the Q-factor, the policy gradient, and the policy function.
Indeed, the majority of existing adaptive optimal control and DP methods \cite{Bertsekas1996a, Si2004, Powell2007, Lewis2013, Jiang2017} falls into this category, and RL is considered as a machine learning reinterpretation of direct adaptive control \cite{Sutton1992}.
The main advantage of these methods is that they are effective in tackling the presence of static uncertainties such as the unknown parameters in the DP algorithm.
As a result, this allows the DP problem to be solved without directly using the knowledge of the underlying system (also known as the environment in RL), i.e., the optimal solution is obtained in a model-free manner.
In spite of its popularity, the adaptive control perspective is not effective in tackling the presence of dynamic uncertainties \cite{Liu2014} in DP algorithms. Such dynamic uncertainty may be caused by coupling the standard DP algorithm with other numerical algorithms, where each of these algorithms then serves as a dynamic uncertainty to the DP algorithm. It may also arise from the decentralized DP problem, where each node in a large-scale network executes its own version of the DP algorithm, and interacts with its neighbors through the outputs and inputs. The algorithm executed in its neighboring nodes can be considered as dynamic uncertainty to the node itself. Existing learning-based DP algorithms are not directly applicable to handle this type of disturbances. 

The second perspective, which we refer to as the ``robust control perspective", aims at strengthening the DP algorithm so that it is robust to the presence of disturbance.
A remarkable feature of this type of methods is that it is effective in dealing with both static and dynamic uncertainties.
Unlike the adaptive control persepctive, the development in this direction is still rudimentary.
Only a few results \cite{Iyengar2005,Nilim2005,Lim2013} are available for  solving DP and RL problems along this track, in which the authors still only considered the static uncertainty caused by the unknown transition probability measures.
Besides, these methods are only available for MDPs.
In other words, there is no robust DP solution for dynamical systems described by differential equations.
As a result, it is still an open problem how to develop DP algorithms that are robust to both static and dynamic uncertainties.
 

In this paper, we propose a novel robust DP theory for continuous-time linear dynamical systems.
Compared with traditional DP and adaptive optimal control, we take a completely different path to investigate DP methods from a viewpoint of nonlinear system theory \cite{Khalil2002} and small-gain theory \cite{Zames1966,Jiang1994}.
As a consequence, we will provide a complete robustness analysis on the DP algorithm, under multiple types of uncertainties, including external disturbance,  dynamic uncertainty, and  stochastic noise, that cannot be dealt with by previously known results.
The proposed robust DP framework is based on the dynamic property of differential matrix Riccati equation (DMRE).
Recall from \cite{Willems1971, Kucera1973} that under observability and stabilizability assumptions, the unique symmetric positive definite solution to the algebraic Riccati equation (ARE) is asymptotically stable for the DMRE, backward in time.
In \Cref{RVI}, we further improve this result by showing that the DMRE also admits a linear $L^2$ gain \cite{Schaft2017} for any arbitrarily large set of initial conditions within the region of attraction, which we will refer to as ``semiglobal gain assignment''.
This conclusion lays the foundation of our small-gain analysis on the continuous-time VI, which in turn leads to a sequence of convergence and robustness results.
A comparison between different DP methods is given in \cref{tab0}.
We admit that one drawback of robust DP is that it requires the nominal value of the components in the algorithm, and hence is not model-free as in existing RL and adaptive optimal control methods.
This drawback can be easily conquered by combining our robust DP with existing adaptive optimal control results.

\begin{table}
\centering
\caption{Comparison between Applications of Different DP Methods}
\begin{tabular}{l |c |c | c}
\cline{1-4}
Application scenarios&DP &   Adaptive DP & Robust DP \\
\cline{1-4}
Ideal case       & Yes                        &   Yes & Yes\\
Static uncertainty       & No                        &   Yes & Yes\\
Dynamic uncertainty & No                     &  No & Yes \\
Model free                & No                      &  Yes & No \\
\cline{1-4}
\end{tabular}
\label{tab0}
\end{table}

To demonstrate the power of the proposed method, in \Cref{sec_applicaiton}, we apply robust DP to solve four classical problems  arising from the field of adaptive optimal control. 
In the first application, we show that the continuous-time VI can be implemented  with system matrices estimated iteratively from a time series. 
The estimation error is treated as an external disturbance, and the convergence of VI is proved via robust DP theory.
In the second application, an improved version of the continuous-time adaptive dynamic programming (ADP) \cite{Bian2016b} is proposed by coupling the recursive least square (RLS) estimation of certain matrix inverse in the ADP learning process.
Robust DP is used to tackle the presence of RLS error.
Compared with existing results, the proposed ADP algorithm is more computationally efficient, as the estimate of the matrix inverse is updated together with the ADP learning.
In the third application, we develop a continuous-time stochastic ADP theory for a class of ergodic control problems, that generalizes the main result of \cite{Bian2016a}.
Different from the stochastic approximation \cite{Kushner2003} and Monte Carlo methods \cite[Chapter 5]{Sutton1998} in traditional RL, a new method for convergence analysis based on robust DP is proposed in the continuous-time setting, due to the complex nature of the continuous-time ergodic control problem.
In our fourth and last application, we propose a novel decentralized VI algorithm for solving coupled AREs.
The state-space based small-gain theory \cite{Jiang1994} is applied with our robust DP framework to provide a sufficient  condition for the convergence analysis of coupled AREs.
This result is especially useful in developing algorithms for robust ADP \cite{Jiang2013} and solving non-zero-sum differential games \cite{Starr1969,Starr1969a}.

To further illustrate the proposed result, we also give three practical simulation examples in \Cref{simulate}.



{\it Notation:} 
Throughout this paper, 
$I_n$ denotes the identity matrix of dimension $n$.  
$\mathbb{R}$ and $\mathbb{R}_+$ denote the set of real numbers and the set of nonnegative real numbers, respectively.
$\mathbb{Z}_+$ denotes the set of nonnegative integers.
$|\cdot|$ denotes the Euclidean norm for vectors, or the induced matrix norm for matrices. 
$\mathcal{S}^n$ denotes the normed space of all  $n$-by-$n$ real symmetric matrices, equipped with the induced matrix norm.
$\mathcal{S}^n_+=\{P\in\mathcal{S}^n:P>0\}$.
For a matrix $M\in\mathbb{R}^{n\times m}$, 
$\ves(M)=[M_1^T,M_2^T, \cdots, M_m^T]^T$, where $M_i\in\mathbb{R}^n$ is the $i$-th column of $M$. 
For any $M\in\mathcal{S}^{n}$, denote $\lambda_m(M)$ and $\lambda_M(M)$ as the minimum and maximum eigenvalues of $M$, respectively; and
$\vecs(M)=[M_{11},M_{12},\cdots,M_{1n},M_{22},M_{23},\cdots, M_{(n-1)n},M_{nn}]^T$,
where $M_{ij}\in\mathbb{R}$ is the $(i,j)$-th element of matrix $M$.
$ \langle \cdot,\cdot\rangle_F$ denotes the Frobenius inner product.
$\otimes$ and $\oplus$ indicate the Kronecker product and Kronecker sum, respectively. 
Given a set $Q$, ${\rm int}(Q)$ denotes the interior of $Q$.
$B_\varepsilon$ denotes an open ball centered at the origin with radius $\varepsilon$.
A  function $f: Q\rightarrow \mathbb{R}_+$, where $Q\subseteq\mathbb{R}^n$ and $0\in Q$, is called positive definite, if $f(x)>0$ for all $x\in Q\setminus\{0\}$, and $f(0)=0$.
{\mymark For $f: \mathbb{R}\rightarrow \mathbb{R}_+$ and $g: \mathbb{R}\rightarrow \mathbb{R}_+$, denote $f(x) = o(g(x))$ if $\lim_{x\rightarrow 0}f(x)/g(x)=0$.}

\section{Preliminaries}\label{sec2}

\subsection{System description}\label{MP_SD}

Consider the following linear time-invariant system:
\begin{align}
\dot x=Ax+Bu,\label{sys_equ1}
\end{align}
where $x\in \mathbb{R}^n$ is the system state,  $u\in\mathbb{R}^m$ is the control input, and $A\in\mathbb{R}^{n\times n}$ and $B\in\mathbb{R}^{n\times m}$ are system matrices.
 Assume $(A,B)$ is stabilizable.

Denote the cost corresponding to system \cref{sys_equ1} as 
\begin{align}
\mathcal{J}(x(0);u)=\int_0^{\infty}(x^TQx+u^TRu)ds, \label{cost_equ1}
\end{align}
where $Q=Q^T\geq0$, $R=R^T>0$, and $(A,Q^{1/2})$ is observable.
It is well known that $\mathcal{J}$ is minimized under the optimal controller $u^*=-K^*x$, where $K^*=R^{-1}B^TP^*$, with $P^*$ the unique symmetric positive definite solution to the following ARE:
\begin{align}
0=A^TP^*+P^*A-P^*BR^{-1}B^TP^*+Q.\label{are_equ1}
\end{align}
Moreover, $A-BK^*$ is Hurwitz.


\subsection{DMRE and continuous-time VI}\label{DMREVI}
Since \cref{are_equ1} is a nonlinear matrix equation, it is not easy to solve $P^*$ from the ARE directly.
One way of finding $P^*$ is to use the continuous-time VI \cite{Bian2016b}.
Before introducing the VI algorithm, we define
a real sequence $\{h_k\}_{k=0}^\infty$ satisfying
\begin{align*}
h_k>0,\quad \lim_{k\rightarrow \infty}h_k=0,\quad \sum_{k=0}^\infty h_k=\infty.
\end{align*}
In addition, denote $\{B_q\}_{q=0}^\infty$ as an increasing real sequence with $B_0>0$ and $\lim_{q\rightarrow\infty}B_q=\infty$.

The continuous-time VI is recalled from \cite{Bian2016b} and shown in  \cref{alg0}. 
Note that if $Q>0$, then the initial choice on $P_0$ can be relaxed to $P_0=P_0^T\geq0$.
Detailed convergence analysis on \cref{alg0}, and its extensions to model-free adaptive optimal controller design, can be found in  \cite{Bian2016b}. 
However, it still remains an open problem how robust  \cref{alg0} is to various types of disturbance.
 As shown in subsequent sections, we will provide the first solution to this fundamentally challenging issue for continuous-time dynamical systems.
 
\begin{algorithm}[t]
\caption{Continuous-time value iteration}
\label{alg0}
\begin{algorithmic}
\STATE Choose $P_0=P_0^T> 0$. Let $\bar\varepsilon>0$ be a small threshold. $k, q\gets 0$.
\LOOP
    \STATE $ P_{k+1/2}\gets P_k+h_k(A^TP_k+P_kA-P_kBR^{-1}B^TP_k+Q)$
    \IF {$P_{k+1/2}>0$ and $| P_{k+1/2}-P_{k}|/h_k<\bar\varepsilon$}
    \RETURN $P_{k}$ as an approximation to $P^*$
    \ELSIF {$| P_{k+1/2}|>B_q$ or $ P_{k+1/2}\not>0$}    
     \STATE $P_{k+1}\gets P_0$. $q\gets q+1$.
     \ELSE
     \STATE $ P_{k+1}\gets  P_{k+1/2}$
    \ENDIF
    \STATE $k\gets k+1$
    \ENDLOOP
 \end{algorithmic}
\end{algorithm}


%


\section{Robust DP and VI for continuous-time systems}\label{RVI}
The purpose of this section is to extend \cref{alg0} in different directions by providing a concrete stability and robustness analysis for the DMRE and VI.
\subsection{Robust DP and DMRE}\label{RVI_DMRE}
As it has been shown in \cite{Willems1971,Kucera1973}, for any $P(0)=P(0)^T\geq0$, the solution to the following DMRE converges to $P^*$ asymptotically as $t$ goes to infinity:
 \begin{align}
 \dot P=A^TP+PA-PBR^{-1}B^TP+Q.\label{dmre_equ1}
 \end{align}
Denoting $K=R^{-1}B^TP$,  we have from \cref{dmre_equ1} that 
 \begin{align*}
 \dot P=&A^TP+PA-PBR^{-1}B^TP+Q\\
 =&(A-BK)^TP+P(A-BK)+K^TRK+Q\\
  =&(A-BK^*)^TP+P(A-BK^*)+(K^*)^TRK^*+Q-(K-K^*)^TR(K-K^*).
   \end{align*}
   Subtracting \cref{are_equ1} from the above equation, and letting $\tilde{P} = P - P^*$, we have
    \begin{align}
 \dot {\tilde P}=(A-BK^*)^T\tilde P+\tilde P(A-BK^*)-\tilde PBR^{-1}B^T\tilde P.\label{DP_equ1}
 \end{align}
 
 The following  two lemmas play an important role in developing our robust VI:
  \begin{lemma}\label{lem1}
$\hat P$ is globally\footnote{Global in the sense that the region of attraction is the entire normed space of all  $n$-by-$n$ real matrices equipped with the induced matrix norm.} exponentially stable at $P^*$, where $\hat P$ is the solution of the following system:
  \begin{align}
 \dot {\hat P}&=(A-BK^*)^T\hat P+\hat P(A-BK^*)+(K^*)^TRK^*+Q,\quad \hat P(0)\in\mathbb{R}^{n\times n}.\label{PE_equ1}
 \end{align}
 \end{lemma}
 \begin{proof}
Denote $\xi=\ves(\hat P-P^*)\in\mathbb{R}^{n^2}$.
Then, by subtracting \cref{are_equ1} from  \cref{PE_equ1}, one has 
  \begin{align}
  \dot {\xi}&=((A-BK^*)\oplus (A-BK^*))^T \xi.\label{lem1_pf_equ1}
 \end{align}
 Since $A-BK^*$ is Hurwitz, $((A-BK^*)\oplus (A-BK^*))^T$ is also Hurwitz \cite{Brewer1978}.
 This completes the proof.
 \end{proof}
 
 \begin{remark}\label{rem0}
 Note from \cref{PE_equ1} that when $\hat P(0)\in\mathcal{S}^n$, we have $\hat P(t)\in\mathcal{S}^n$ for all $t>0$.
 Since $\mathcal{S}^n\subset \mathbb{R}^{n\times n}$, we know $\hat P$ is also exponentially stable at $P^*$ in  $\mathcal{S}^n$.
 \end{remark}

 \begin{lemma}\label{lem0}
 Consider a dynamical system defined on the {\highlight inner product space $(\mathcal{S}^n,\langle\cdot,\cdot\rangle_F)$}:
 \begin{align}
 \dot P=G(P),\label{lem0_equ1}
 \end{align}
 where $G:\mathcal{S}^n\rightarrow\mathcal{S}^n$ is locally Lipschitz, and satisfies $G(0)=0$.
 If $R_A$ is the region of attraction of the origin for system \cref{lem0_equ1}, 
 then there exists a smooth Lyapunov function $V: R_A\rightarrow \mathbb{R}_+$, such that 
\begin{align*}
 \langle\partial_xV(P),G(P)\rangle_F&<0, \quad V(P)>0\quad \forall P \in R_A \setminus \{0\},\\
\lim_{P\rightarrow \partial R_A}V(P)&= \infty,\quad \langle V(0), G(0)\rangle_F=0,\quad V(0)=0.
\end{align*}
 \end{lemma}

\begin{proof}
Denote a mapping $\mathcal{M}(\cdot):\mathcal{S}^n\rightarrow\mathbb{R}^{n(n+1)/2}$, such that 
\begin{align*}
\mathcal{M}(M)=[M_{11},\sqrt{2}M_{12},\cdots,\sqrt{2}M_{1n},M_{22},\sqrt{2}M_{23},\cdots, \sqrt{2}M_{(n-1)n},M_{nn}]^T.
\end{align*}
Then, for any $M_1,M_2\in\mathcal{S}^n$, $\mathcal{M}^T(M_1)\mathcal{M}(M_2) = \langle M_1,M_2\rangle_F$.
Hence, $\mathcal{M}(\cdot)$ is a smooth isometric isomorphism\footnote{
A bounded linear operator is called an isometric isomorphism if it is a norm preserving bijection which is continuous and has a continuous inverse \cite[pp. 71]{Reed1980}.}.
Then, one can rewrite \cref{lem0_equ1} as the following ODE:
\begin{align}
\dot p=g(p),\label{lem0_equ2}
\end{align}
where $p=\mathcal{M}(P)$, and $g=\mathcal{M} \circ G\circ\mathcal{M}^{-1}$.
Denote the region of attraction of $0\in \mathbb{R}^{n(n+1)/2}$ by $R_A'\subseteq \mathbb{R}^{n(n+1)/2}$.
Since $R_A$ is not empty, $R_A'$ is also not empty.
By converse Lyapunov theorem \cite[Theorem 4.17]{Khalil2002}, we know there exists a smooth function $W(\cdot): R_A'\rightarrow \mathbb{R}_+$, such that 
\begin{align*}
\partial_x W(p)g(p)&<0, \quad W(p)>0\quad \forall p \in R_A' \setminus \{0\},\\
\lim_{p\rightarrow \partial R_A'}W(p)&= \infty,\quad \partial_x W(0)g(0)=0,\quad W(0)=0.
\end{align*}

We claim $R_A' = \mathcal{M}(R_A)$.
Otherwise, if there exist $P_0\in R_A$ and $\mathcal{M}(P_0)\not\in R_A'$, then $R_A'$ is no longer the region of attraction for \cref{lem0_equ2} since the solution to \cref{lem0_equ2}  starting from $\mathcal{M}(P_0)$ also converges to the origin, by the norm preserving property of $\mathcal{M}$.
Similarly, if there exists $p_0\in R_A'$ such that $\mathcal{M}^{-1}(p_0)\not\in R_A$,  then $R_A$ is no longer the region of attraction for \cref{lem0_equ1}.

Now, we  define a function $V(\cdot): R_A\rightarrow \mathbb{R}_+$, such that $V = W\circ\mathcal{M}$.
By the definition of matrix calculus, $\partial_xV = \mathcal{M}^{-1}\circ\partial_xW \circ\mathcal{M}$.
It is easy to see that  all the higher-order derivatives of $V$ can be defined in a  similar manner.
Hence,  $V$ is also smooth.
By the definition of Frobenius inner product, 
\begin{align*}
\partial_x W(p)g(p) = \langle\partial_xV(P),G(P)\rangle_F, \quad \forall P\in \mathcal{S}^n.
\end{align*}
This concludes the proof.
\end{proof}

\begin{remark}\label{rem_lem0}
\cref{lem0} extends the converse Lyapunov theorem for general nonlinear systems \cite[Theorem 4.17]{Khalil2002} to the space of real symmetric matrices. 
{\highlight The converse statement of \cref{lem0}, i.e., the Lyapunov theorem for the stability of general nonlinear systems over  $(\mathcal{S}^n,\langle\cdot,\cdot\rangle_F)$, can also be derived in a similar way.}
Moreover, one can also generalize the converse Lyapunov theorem for exponentially stable systems \cite[Theorem 4.14]{Khalil2002}.
We omit this direct extension to avoid duplication.
\end{remark}

 \begin{proposition}\label{thm1}
 $P$ is exponentially stable at $P^*$ over $\mathcal{S}^n$.
 \end{proposition}
 \begin{proof}
 Note from \cref{lem1} and \cref{rem0} that for any $\hat P(0)\in\mathcal{S}^n$, ${\hat P}$ converges exponentially to $P^*$.
 Then, following \cref{lem0}, \cref{rem_lem0}, and \cref{lem1_pf_equ1}, there exists a smooth Lyapunov function $V: \mathcal{S}^n\rightarrow \mathbb{R}_+$ satisfying
 \begin{align*}
 C_1|\hat P-P^*|^2\leq V(\hat P-P^*)\leq  C_2|\hat P-P^*|^2,\\
\dot V(\hat P-P^*)\leq-C_3 |\hat P-P^*|^2,\quad  |\partial_x V(\hat P-P^*)|<C_4|\hat P-P^*|,
 \end{align*}
{\highlight for any $\hat P$ on $\mathcal{S}^n$}, where $C_i>0$, $i=1,2,3,4$.
 Note that we use the induced norm here instead of the Frobenius norm, due to the equivalence of matrix norms.
 
Comparing the dynamics of $ \hat P$ and $ P$, we see the only difference between these two systems is the quadratic term $\tilde PBR^{-1}B^T\tilde P$.
Now, by taking the derivative of $V$ along the solutions of system \cref{DP_equ1}, we have
  \begin{align*}
\dot V(\tilde P)\leq-C_3 |\tilde P|^2+C_5 |\partial_x V(\tilde P)||\tilde P|^2\leq-C_3 |\tilde P|^2+C_4C_5|\tilde P|^3,
 \end{align*}
 {\highlight for any $\tilde P\in\mathcal{S}^n$ and constant $C_5>0$.}
 From the above inequality, we know there exist $\varepsilon>0$ and $C_6>0$, such that, 
   \begin{align}
\dot V(\tilde P)\leq-C_6 |\tilde P|^2\leq -\frac{C_6}{C_2}V(\tilde P),\quad \forall |\tilde P|<\varepsilon.\label{thm1_pf2}
 \end{align}
The proof is then completed using the Lyapunov theorem \cite[Theorem 4.10]{Khalil2002}.
 \end{proof}
 
  \begin{remark}\label{rem1}
Compared with \cite[Remark 21]{Willems1971} and \cite[Theorem 17]{Kucera1973}, \cref{thm1} provides a stronger result, in the sense that it characterizes the convergence speed of $\tilde P(t)$ in a neighborhood of the origin. 
This is the foundation of our robustness analysis on DMRE and VI.
 \end{remark}
 

We will exploit the important feature of exponential stability further in the rest  of this paper.
First,  let us consider the following variant of \cref{dmre_equ1} subject to a disturbance input $\Delta(t)=\Delta^T(t)$:
  \begin{align}
 \dot P_\Delta&=A^TP_\Delta+P_\Delta A-P_\Delta BR^{-1}B^TP_\Delta+Q+\Delta,\quad P_\Delta(0)=P_\Delta^T(0)\geq0.\label{dmre_equ2}
 \end{align}
 
 \begin{remark}
$\Delta$ can represent a large class of disturbances. 
 In particular, we conduct robustness analysis on \cref{dmre_equ2} in \cref{thm2} below by considering three different forms of $\Delta$, including 
\begin{enumerate*}[label=\textit{\alph*)}]
\item a state-independent external signal (\cref{thm2}, parts (i) and (ii)); 
\item the output of a nonlinear dynamical system  (\cref{thm2}, part (iii)); and 
\item a stochastic disturbance (\cref{thm2}, part (iv)).
\end{enumerate*}
The assumption $\Delta(t)=\Delta^T(t)$ is to guarantee that $P_\Delta$ is always symmetric.
This condition can be easily satisfied in practice, since for any $M\in\mathbb{R}^{n\times n}$, $x^TMx=\frac{1}{2}x^T(M+M^T)x$, and $\frac{1}{2}(M+M^T)$ is real symmetric.
 \end{remark}
 
 \begin{theorem}\label{thm2}
 Consider system \cref{dmre_equ2} with $Q>0$.
{\highlight  Denoting $\tilde P_\Delta = P_\Delta-P^*$,} we have
\begin{tasks}[counter-format=(tsk[r]),label-width=4.2ex](1)
 \task  If $\inf_t\lambda_m(Q+\Delta(t))\geq0$ and $\sup_t\lambda_M(Q+\Delta(t))<\infty$, then $P_\Delta$ is well defined on $\mathbb{R}_+$, and there exists  $M\in\mathcal{S}^n$ that is dependent on $P_\Delta(0)$, such that $0\leq P_\Delta(t)<M$ for all $t>0$.
 \task  If $\Delta$ satisfies the conditions in (i), and $\lim_{t\rightarrow \infty}\Delta(t)=0$, then $\lim_{t\rightarrow\infty}P_\Delta(t)=P^*$.
If in addition $\Delta\in L^2$, then {\highlight $\tilde P_\Delta\in L^2$.}
  \task  There exists $\gamma>0$, such that if the following system\footnotemark[3]
  \begin{align}
  \dot M=f(M,P_\Delta),\quad \Delta(t)=\Delta(P_\Delta,M),\label{thm2_equ1}
  \end{align}
  where 
$f$ {\highlight and $\Delta$ are} locally Lipschitz, $f(M^*,P^*)=0$, and $\Delta(P^*,M^*)=0$,
is zero-state detectable\footnotemark[4] and admits an IOS Lyapunov function  {\highlight $V_f$} satisfying
   \begin{align}
{\highlight \dot V_f(\tilde M)\leq- |\Delta|^2+\gamma^2| \tilde P_\Delta|^2}, \quad \forall M\in B_{\varepsilon_0}(M^*),\ \varepsilon_0>0,\label{thm2_equ2}
 \end{align}
{\highlight where $\tilde M = M-M^*$,}
then $(P_\Delta,M)$ is asymptotically stable at $(P^*,M^*)$.
\task  Suppose $\Delta(t)=\sum_{i=1}^N\Delta_i (P_\Delta)v_i(t)$, where $N>0$, $\Delta_i:\mathcal{S}^n\rightarrow\mathcal{S}^n$, and the $v_i$ are one-dimensional i.i.d. Gaussian white noises.
Then, there exists $\gamma>0$, such that if  {\highlight  $\sum_i|\Delta_i|^2<\gamma|\tilde P_\Delta|$} in a neighborhood of $P^*$,  $P_\Delta$ is asymptotically stable at $P^*$ in the mean square sense over $\mathcal{S}^n$.
\end{tasks}
 \end{theorem}
 
 \footnotetext[3]{$M$ can be either a real vector or a real matrix, depending on the specific problem formulation.
 For  consistency, here we consider $M$ as a real matrix of an appropriate dimension.}
 \footnotetext[4]{Here, with slight abuse of notation, we say \cref{thm2_equ1} is zero-state detectable if $\Delta\equiv0$ and $P_\Delta\equiv P^*$ imply $M\equiv M^*$.}
 
 \begin{proof}
To prove part (i), we first introduce the following finite-horizon cost: 
 \begin{align*}
\mathcal{J}_{t}(x(t);u,Q)=x^T(0)P_\Delta(0)x(0)+\int_{t}^{0}(x^T(s)Q(s)x(s)+u^T(s)Ru(s))ds,
\end{align*}
where $t<0$ is an arbitrary time instant, and $Q(s)=Q+\Delta(s)$. 
Since $Q(s)\geq0$ on $[t,0]$, it is well known from the LQR theory \cite[Chapter 6.1]{Liberzon2012} that $\inf_u\mathcal{J}_{t}(x(t);u,Q)=x^T(t)M(t)x(t)$, where $M(s)=M^T(s)>0$, $s\in[t,0]$, satisfies
  \begin{align*}
 -\dot M&=A^TM+M A-M BR^{-1}B^TM+Q+\Delta,\quad M(0)=P_\Delta(0).
 \end{align*}
Moreover, the optimal controller for $\mathcal{J}_{t}$ is $u^o(s):=-R^{-1}B^TM(s)$.

 On the other hand, we have from the conditions on $Q+\Delta$ that there exists a constant matrix $\overline Q\in\mathcal{S}$, such that $0\leq Q(s)<\overline Q$ for all $s$.
  Thus,
 \begin{align}
0\leq x^T(t)M(t)x(t)
\leq x^T(0)P_\Delta(0)x(0)+\int_{t}^{0}(x^T Q(s)x+(\bar u^o)^TR\bar u^o)ds\notag\\
\leq x^T(0)P_\Delta(0)x(0)+\int_{t}^{0}(x^T\overline Qx+(\bar u^o)^TR\bar u^o)ds,\label{thm1_pf_equ1}
\end{align}
where $\bar u^o:=\arg\inf_u\mathcal{J}_{t}(x(t);u,\overline Q)$.
Since  $\overline Q$ is positive definite, we know there exists a real symmetric matrix $\overline M>0$, such that 
\begin{align*}
\inf_u\mathcal{J}_{t}(x(t);u,\overline Q)< x^T(t)\overline Mx(t).
\end{align*}
Then, we have from \cref{thm1_pf_equ1} that $0\leq M(t)<\overline M$ for all $t<0$.
Comparing the definitions of $M$ and $P_\Delta$, we know $M(t)=P_\Delta(-t)$.
Thus, $0\leq P_\Delta(t)<\overline M$ for all $t>0$.

To prove part (ii), note from part (i) that $P_\Delta$ is bounded on $\mathbb{R}_+$.
Then, since $P(t)$ converges to $P^*$,  for any $\varepsilon>0$, there exists $T_0>0$, such that $\sup_{T>T_0}|P(t+T)-P^*|<\varepsilon$, given $P(t)=P_\Delta(t)$  for any $t>0$.
On the other hand, by \cite[Theorem 55]{Sontag1998}, for any $T_1>0$ and $\varepsilon>0$, we can find $t_0>0$ under which $\sup_{t\geq t_0}|\Delta(t)|$ is sufficiently small,  so that  $\sup_{T\in[0,T_1]}|P(t+T)-P_\Delta(t+T)|<\varepsilon$, given $P(t)=P_\Delta(t)$  for all $t>t_0$.

Now, by picking $T_1=2T_0$, one can guarantee from the above analysis that $|P^*-P_\Delta(t+T)|<2\varepsilon$ for all $t>t_0$ and $T\in[T_0,2T_0]$.
Thus, we know $\sup_{t>t_0+T_0}|P_\Delta(t)-P^*|\leq2\varepsilon$.
Since $t_0$ exists for any $\varepsilon$, which can be made arbitrarily small, we have  $\lim_{t\rightarrow\infty}P_\Delta(t)=P^*$.

Moreover, choosing the same Lyapunov function in the proof of \cref{thm1}, we know there exist positive constants $C_1$, $C_2$, and $\varepsilon_1$, such that 
   \begin{align}
\dot V(\tilde P_\Delta)\leq-C_1 |\tilde P_\Delta|^2 + C_2|\tilde P_\Delta||\Delta|,\quad \forall |\tilde P_\Delta|<\varepsilon_1,\label{thm2_pf2}
 \end{align}
 where $\tilde P_\Delta = P_\Delta-P^*$.
  By completing the squares, we have from \cref{thm2_pf2} that \cref{dmre_equ2} admits a finite linear $L^2$ gain in a neighborhood of $P^*$.
 Thus, by $H^\infty$ control theory, $\tilde P_\Delta\in L^2$ if $\Delta\in L^2$.

Now, we prove part (iii).
Note from \cref{thm2_equ2} and  \cref{thm2_pf2} that 
 if  
\begin{align*}
\gamma<\frac{C_1}{\sqrt{2}C_2},
 \end{align*}
 then {\highlight by defining $\bar V(P,M) = V(P)+\frac{C_2^2}{C_1} V_f(M)$,}
    \begin{align*}
\frac{d}{dt}{\highlight\bar V(\tilde P_\Delta,\tilde M)} &\leq-C_1 |\tilde P_\Delta|^2 + C_2|\tilde P_\Delta||\Delta|- \frac{C_2^2}{C_1}|\Delta|^2+\frac{C_2^2}{C_1}\gamma^2|\tilde P_\Delta|^2\\
&=-\left(\frac{C_1}{2}-\frac{C_2^2}{C_1}\gamma^2\right)|\tilde P_\Delta|^2-\frac{C_2^2}{2C_1}|\Delta|^2\\
&\leq-C_3 |\tilde P_\Delta|^2-\frac{C_2^2}{2C_1}|\Delta|^2,\quad \forall |\tilde P_\Delta|<\varepsilon_1,\ {\highlight |\tilde M|< \varepsilon_0,}
 \end{align*}
for some $C_3>0$.
Since \cref{thm2_equ1} is zero-state detectable, we have from LaSalle's invariance principle \cite[Corollary 4.1]{Khalil2002} that $(P_\Delta,M)$ is asymptotically stable at $(P^*,M^*)$.

Finally, to prove part (iv) involving stochastic disturbance,  from It{\^o}'s lemma \cite{Steele2001} and \cref{thm1_pf2}, it follows that
   \begin{align*}
\mathcal{L} V(\tilde P_\Delta)\leq-C_1 |\tilde P_\Delta|^2 + C_4\sum_{i=1}^N|\Delta_i|^2,\quad \forall |\tilde P_\Delta|<\varepsilon,
 \end{align*}
 for some positive constants $\varepsilon$ and $C_4$, where $\mathcal{L}$ denotes the differential generator.
 Note that $C_4$ is bounded since $\partial_x^2V$ is  bounded on any compact sets, as $V$ is smooth.
 Obviously, if
\begin{align*}
 \sum_{i=1}^N|\Delta_i|^2<\frac{C_1}{C_4}|\tilde P_\Delta|^2,
 \end{align*}
 then
\begin{align*}
\mathcal{L}  V(\tilde P_\Delta)\leq-C_5 |\tilde P_\Delta|^2,\quad \forall |\tilde P_\Delta|<\varepsilon,
 \end{align*}
for some $C_5>0$.
This concludes the proof.
 \end{proof}
 
\cref{thm1} and \cref{thm2} imply that the DMRE behaves very similar to an exponentially stable linear system in a  neighborhood of $P^*$, and thus exhibits a series of nice properties.
However, the stability and robustness results in these two theorems are of limited use in practice, since they hold only in a neighborhood of $P^*$.
In order to obtain desirable transient performance for the DMRE in a sufficiently large compact set, we need to  design carefully the cost \cref{cost_equ1}.
Indeed, the following corollary shows that by choosing $Q$ and $R$ properly, we can guarantee the {\it semi-global} exponential stability of \cref{dmre_equ1} at $P^*$. By "semi-global", we mean that the domain of attraction is bounded but can be made as large as possible \cite{Sastry1999}.

  \begin{corollary}\label{thm3}
Given $Q_0=Q_0^T>0$ and $R_0=R_0^T>0$, for any compact set $\mathcal{S}_0\subset\mathcal{S}^n_+$, there exists a constant $\lambda>0$, such that by choosing $Q=\lambda Q_0$ and $R=\lambda R_0$, each trajectory of \cref{dmre_equ1} starting at $P(0)\in\mathcal{S}_0$ converges exponentially to $P^*$.
 \end{corollary}
 \begin{proof}
 First, note that under the choice of $Q=\lambda Q_0$ and $R=\lambda R_0$, $K^*$ is independent of $\lambda$, as both $Q$ and $R$ are derived from $Q_0$ and $R_0$ by multiplying the same scaling factor.
Moreover, $P^*$ is a linear function of $\lambda$, and $\lim_{\lambda\rightarrow0^+}P^*=0$.
Now, for any $\mathcal{S}_0$, we can find a small enough $\lambda>0$, such that $P^*<P(0)$ for all $P(0)\in\mathcal{S}_0$.
Then, by choosing $\hat P(0)=P(0)$ in \cref{PE_equ1}, we have for any given $t>0$ and $x(-t)\in\mathbb{R}^n$,
 \begin{align*}
x^T(-t)P(t)x(-t)=\inf_u\left\{x^T(0)P(0)x(0)+\int_{-t}^{0}(x^T Qx+u^TRu)ds\right\}\notag\\
\leq (x^*(0))^TP(0)x^*(0)+\int_{-t}^{0}(x^*)^T (Q+(K^*)^TRK^*)x^*ds=x^T(-t)\hat P(t)x(-t),
\end{align*}
where $x^*$ is the solution to system \cref{sys_equ1} with $u=-K^*x^*$ and $x^*(-t)=x(-t)$.
Moreover, by monotonicity \cite[Lemma 1]{Bian2016f}, $P^*\leq P(t)$ for all $t$.
Since by \cref{lem1} $\hat P(t)$ converges to $P^*$ exponentially, $x^T\hat Px$ also converges to $x^TP^*x$ exponentially for all $x$.
Using $x^TP^*x\leq x^TPx\leq x^T\hat Px$, we know $x^T Px$ converges to $x^TP^*x$ exponentially.
Noting that this is true for all $x$, $P$ thus converges to $P^*$ exponentially.
This completes the proof.
 \end{proof}
 
 \begin{remark}
It is easy to see from \cref{thm3} that although multiplying the same scalar to $Q_0$ and $R_0$ does not influence the optimal feedback gain matrix, the transient performance of the DMRE can be  quite different.
{\highlight Given any $P(0)$, by \cref{thm1} and the converse Lyapunov theorem \cite[Theorem 4.14]{Khalil2002}, we can find a Lyapunov function $V$ satisfying
 \begin{align*}
 C_1|\tilde P|^2\leq V(\tilde P)\leq  C_2|\tilde P|^2,\quad
\dot V(\tilde P)\leq-C_3 |\tilde P|^2,\quad  |\partial_x V(\tilde P)|<C_4|\tilde P|,
 \end{align*}
 where  $C_i>0$, $i=1,2,3,4$,
over a  connected compact set including $P(0)$ and $P^*$.}
As a result, \cref{thm3} allows us to extend the result obtained in \cref{thm2} to any compact sets containing $P^*$ in $\mathcal{S}^n_+$.
 \end{remark}
 
If we are allowed to have more freedom on choosing $Q$ and $R$, it is possible to have the following semi-global gain assignment result:
  \begin{corollary}\label{thm4}
  Given $Q_0=Q_0^T>0$ and $R_0=R_0^T>0$, if $B$ has full rank, then for any $\varepsilon>0$ and $\gamma>0$,  there exists $\lambda>0$, such that \cref{dmre_equ2}  admits a finite linear $L^2$ gain from $\Delta$ to $\tilde P_\Delta$ less than or equal to $\gamma$ {\highlight for  $P(0)\in\{ P\in\mathcal{S}^n_+: P\in B_\varepsilon(P^*)\}$}, with $Q=\lambda Q_0$ and $R=o(\lambda) R_0$.
 \end{corollary}

\begin{proof}
Since $Q_0$ and $R_0$ are multiplied by different scaling factors, different from \cref{thm3}, $K^*$ depends on $\lambda$ here.
Hence, the first step of our proof is to characterize the influence of $\lambda$ on the eigenvalues of $A-BK^*$.

Note that choosing $Q=\lambda Q_0$ and $R=o(\lambda) R_0$ is equivalent to choosing $Q=Q_0$ and $R=\delta_\lambda R_0$, where $\delta_\lambda=o(\lambda)/\lambda$, in the sense that these two choices lead to the same optimal controller.
Denoting $P^*_\lambda$ as the solution to \cref{are_equ1} with $Q=Q_0$ and $R=\delta_\lambda R_0$, we have 
\begin{align}
(A-BK^*_\lambda)^TP^*_\lambda&+P^*_\lambda(A-BK^*_\lambda)=\notag\\
&-\left(Q_0+\left(\sqrt{\delta_\lambda^{-1}} P^*_\lambda\right) BR_0^{-1}B^T\left(\sqrt{\delta_\lambda^{-1}} P^*_\lambda\right)\right),\label{thm4_pf1}
\end{align}
where $K^*_\lambda=\delta_\lambda^{-1} R_0^{-1}B^TP_\lambda^*$.
Since $B$ has full rank, we know from \cite[(40)]{Kwakernaak1972} that there exists $\bar P\in\mathcal{S}^n$, such that $\lim_{\lambda\rightarrow 0}\sqrt{\delta_\lambda^{-1}}P^*_\lambda=\bar P$.
Thus, for any two  positive  constants $C$ and $\varepsilon$, we can choose a small enough $\lambda$, such that $\left|\sqrt{\delta_\lambda^{-1}} P^*_\lambda-\bar P\right|<\varepsilon$ and
\begin{align*}
\sqrt{\delta_\lambda^{-1}}\left(Q_0+\left(\sqrt{\delta_\lambda^{-1}} P^*_\lambda\right) BR_0^{-1}B^T\left(\sqrt{\delta_\lambda^{-1}} P^*_\lambda\right)\right)>C I_n.
\end{align*}
This, together with the Lyapunov equation \cref{thm4_pf1}, implies that for any $\alpha>0$, we can find $\lambda>0$, such that
\begin{align*}
(A-BK^*_\lambda)^TM&+M(A-BK^*_\lambda)<-\alpha M
\end{align*}
for some constant matrix $M=M^T>0$.
This implies that the  eigenvalues of $A-BK^*_\lambda$ can be placed arbitrarily far to the left from the imaginary axis, by choosing a small enough $\lambda$.



Now, by  the linear matrix inequality argument \cite[Lemma 4.1]{Gahinet1994}, we know that for any $\gamma>0$, one can find a $\lambda>0$, such that the following system admits a linear $L^2$ gain from $v$ to $\xi$ less than or equal to $\gamma$:
  \begin{align*}
  \dot {\xi}=((A-BK_\lambda^*)\oplus (A-BK_\lambda^*))^T \xi+v.
 \end{align*}
By following similar Lyapunov theorem arguments in the proof of \cref{thm1},
we know that for any $\varepsilon>0$, 
the $L^2$ gain of system \cref{dmre_equ2} can be made arbitrarily small on $\{ P\in\mathcal{S}^n_+: P\in B_\varepsilon(P^*)\}$, by choosing a sufficiently small $\lambda$.
This completes the proof.
 \end{proof}

 \begin{remark}
The full-rank condition on $B$ is required to satisfy the matching condition, which is a common assumption in nonlinear gain assignment and robust control literature \cite{Jiang1994,Praly1996,Isidori1999,Liu2014}. 
To relax this assumption in the case of unmatched disturbance, one way is to study
cascaded systems with full rank input matrices via recursive backstepping \cite{Krstic1995}, or a combination of the backstepping and  small-gain approaches \cite{Liu2014}. 
 \end{remark}
 
 
  The following corollary is a direct extension of \cref{thm2}, parts (iii) and (iv), and \cref{thm4}, and thus its proof is omitted.
 
  \begin{corollary}\label{col1}
   Given $Q_0=Q_0^T>0$, $R_0=R_0^T>0$, and $\lambda>0$, define $Q=\lambda Q_0$ and $R=o(\lambda) R_0$.
Suppose $B$ has full rank.
\begin{tasks}[counter-format=(tsk[r]),label-width=4.2ex](1)
\task 
For any $\gamma>0$, if system \cref{thm2_equ1} satisfies the conditions in \cref{thm2}, part (iii), then there exist  $\lambda>0$, such that $(P_\Delta,M)$ is  asymptotically stable at $(P^*,M^*)$.
\task For any  $\gamma>0$ and $\varepsilon>0$, if $\Delta$ satisfies the definition in \cref{thm2}, part (iv), 
then there exists $\lambda>0$, such that $P_\Delta$ is  asymptotically stable at $P^*$ in the mean square sense.
\end{tasks}
 \end{corollary}
 


 \subsection{Robust VI algorithm}
In this subsection, we formally introduce the robust VI (\cref{alg1}) based on the theoretical results in \Cref{RVI_DMRE}.
Note that different from \cref{alg0}, \cref{alg1} includes both a deterministic perturbation term $\Delta_k$ and a stochastic noise term $W_k$ in the updating equation of $P_k$.
\begin{algorithm}[t]
\caption{Continuous-time robust VI}
\label{alg1}
\begin{algorithmic}
\STATE Choose $P_0=P_0^T\geq 0$. $k, q\gets 0$.
\LOOP
    \STATE $ P_{k+1/2}\gets P_k+h_k(A^TP_k+P_kA-P_kBR^{-1}B^TP_k+Q+\Delta_k+W_k)$
        \IF {$P_{k+1/2}>0$ and $| P_{k+1/2}-P_{k}-h_k(\Delta_k+W_k)|/h_k<\bar\varepsilon$}
    \RETURN $P_{k}$ as an approximation to $P^*$
    \ELSIF {$| P_{k+1/2}|>B_q$ or $ P_{k+1/2}\not>0$}
     \STATE $P_{k+1}\gets P_0$. $q\gets q+1$.
     \ELSE
     \STATE $ P_{k+1}\gets  P_{k+1/2}$
    \ENDIF
    \STATE $k\gets k+1$
    \ENDLOOP
 \end{algorithmic}
\end{algorithm}
 
The following theorem shows that \cref{alg1} inherits the robustness property from \cref{dmre_equ2}.
\begin{theorem}\label{thm5}
%
Denote a complete probability space $(\Omega,\mathcal{F}, \mathbb{P})$ equipped with a filtration $\{\mathcal{F}_k\}_{k\in\mathbb{Z}_+}$.
Suppose $Q>0$, $W_k$ is $\mathcal{F}_k$-adapted, $h_k$ is a sequence satisfying the conditions in \Cref{DMREVI}, and $\sum_{k=0}^\infty h_kW_k$ converges with probability one.
Given $\{P_k\}_{k=0}^\infty$ defined in \cref{alg1}, we have  with probability one that,
\begin{tasks}[counter-format=(tsk[r]),label-width=4.2ex](1)
 \task there exist ${\highlight \delta_0}>0$, $N\geq0$, and a compact set $\mathcal{S}_0\subset \mathcal{S}^n_+$ with nonempty interior and $P^*\in \mathcal{S}_0$, such that if $|\Delta_k|<{\highlight \delta_0(1+|P_k|)}$, then $\{P_k\}_{k=N}^\infty\subset \mathcal{S}_0$.
 \task if $\lim_{k\rightarrow \infty}\Delta_k=0$ {\highlight uniformly on any compact set in $\mathcal{S}^n$}, then $\lim_{k\rightarrow\infty}P_k=P^*$.
  \task {\mymark\highlight if $\Delta_k:=\Delta(P_k,M_k)$ is the output to the following updating equation:
  \begin{align}
  M_{k+1}=M_k+h_kf(M_{k},P_k)+Z_k,\label{thm5_equ1}
  \end{align}
  where $\{M_k\}_{k=0}^\infty$ is bounded in $B_{\varepsilon_0}(M^*)$ under a projection term $Z_k$, 
then there exists $\gamma>0$, such that if the conditions in part (iii) of \cref{thm2} are satisfied, 
we have $\lim_{k\rightarrow\infty}(P_k,M_k)=(P^*,M^*)$ locally.}
%
\end{tasks}
\end{theorem}

\begin{proof}
Before proving the part (i), we denote an operator $\mathcal{R}:\mathcal{S}^n\rightarrow \mathcal{S}^n$, such that 
\begin{align*}
\mathcal{R}(P)=A^TP+PA-PBR^{-1}B^TP+Q. 
\end{align*}

Suppose $P_0\neq P^*$.
By \cref{lem0},
we know there exists a smooth Lyapunov function $\mathcal{V}: R_A\rightarrow \mathbb{R}_+$, where $R_A\subset\mathcal{S}^n$ is the region of attraction of $P^*$,  such that 
\begin{align*}
\langle\partial_x \mathcal{V}(P),\mathcal{R}(P)\rangle_F&<0, \quad \mathcal{V}(P)>0,\quad \forall P \in R_A \setminus \{P^*\},\\
\lim_{P\rightarrow \partial R_A}\mathcal{V}(P)&= \infty,\quad \langle\partial_x \mathcal{V}(P^*),\mathcal{R}(P^*)\rangle_F=0,\quad \mathcal{V}(P^*)=0.
\end{align*}
Note that $\mathcal{V}$ defined here is different from the Lyapunov function used in the proof of  \cref{thm1}.
As a result, $\{P:\mathcal{V}(P)\leq C\}$ is a compact subset of $R_A$, for all $C>0$.
Then, there exist $C_0>0$ and $C_1>0$, such that $C_0<\mathcal{V}(P_0)<C_1$.
Furthermore, we can find a sufficiently small constant ${\highlight \varepsilon_\delta}>0$, such that  for all $|\zeta|<{\highlight \varepsilon_\delta}$,
\begin{align}
 \sup_{ \{P:C_0\leq \mathcal{V}(P)\leq C_1\}} \{\langle\partial_x \mathcal{V}(P),(\mathcal{R}(P)+\zeta)\rangle_F\}=- \delta,\label{thm5_pf_equ1}
\end{align}
for some $\delta>0$.

By contradiction, suppose $\{P_k\}_{k=0}^\infty$ is unbounded.
Then, there exists an up-crossing interval $[C_2,C_3]$, with $\mathcal{V}(P_0)<C_2<C_3<C_1$, such that $\{\mathcal{V}(P_k)\}_{k=0}^\infty$ crosses this interval from below infinitely many times.

From the conditions on $W_k$, we know there exists  $E\in\mathcal{F}$ with $\mathbb{P}(E)=1$, such that for all $\omega\in E$, $\{W_k(\omega)\}_{k=0}^\infty$ is bounded.
Fixing $\omega\in E$, we can define two subsequences $\{P_{k_j}\},\{P_{k_j'}\}\subset\{P_k\}$, such that 
\begin{align}
\mathcal{V}(P_{k_j-1})<C_2\leq \mathcal{V}(P_m)<C_3<\mathcal{V}(P_{k_j'}),\ \forall k_j\leq m<k_j'.\label{thm5_pf_equ4}
\end{align}
Choose a sufficiently small $\varepsilon>0$, such that for any $P\in \{P_{k_j}\}$, $B_\varepsilon(P)\subset \{P\in\mathcal{S}^n_+:\mathcal{V}(P)<C_1\}$.
Suppose $q$ is sufficiently large. 
Then,  for any $j\in\mathbb{Z}_+$,
\begin{align}
\varepsilon<|P_{L_\varepsilon(j)}-P_{k_j}|&=\left|\sum_{i=k_j}^{L_\varepsilon(j)-1}h_i(\mathcal{R}(P_i)+\Delta_i+W_i)\right|\notag\\
&\leq\sum_{i=k_j}^{L_\varepsilon(j)-1}h_i(|\mathcal{R}(P_i)|+|\Delta_i|+|W_i|)\leq {\highlight \varepsilon_C}\sum_{i=k_j}^{L_\varepsilon(j)-1}h_i,\label{thm5_pf_equ2}
\end{align}
where $L_\varepsilon(j)=\inf\{i\geq k_j:|P_i-P_{k_j}|>\varepsilon\}$, and  ${\highlight \varepsilon_C}>0$ is a constant independent of $j$.


Then, by \cref{thm5_pf_equ1} and the assumption on $W_k$, one has
\begin{align*}
&\mathcal{V}(P_{L_\varepsilon(j)})-\mathcal{V}(P_{k_j})\\
=&\int_0^1\langle\partial_x\mathcal{V}(P_{k_j}+t(P_{L_\varepsilon(j)}-P_{k_j})),(P_{L_\varepsilon(j)}-P_{k_j})\rangle_Fdt\\
=&\langle\partial_x\mathcal{V}(P_{k_j}),(P_{L_\varepsilon(j)}-P_{k_j})\rangle_F\\
&+\int_0^1\int_0^1\langle\frac{d}{ds}\partial_x\mathcal{V}(P_{k_j}+st(P_{L_\varepsilon(j)}-P_{k_j})),(P_{L_\varepsilon(j)}-P_{k_j})\rangle_Fdsdt\\
=&\sum_{i=k_j}^{L_\varepsilon(j)-1}h_i\langle\partial_x\mathcal{V}(P_{k_j}),(\mathcal{R}(P_{k_j})+\bar \Delta_{i,j})\rangle_F
+\langle\partial_x\mathcal{V}(P_{k_j}),\sum_{i=k_j}^{L_\varepsilon(j)-1}h_iW_i\rangle_F\\
&+\int_0^1\int_0^1\langle\frac{d}{ds}\partial_x\mathcal{V}(P_{k_j}+st(P_{L_\varepsilon(j)}-P_{k_j})),(P_{L_\varepsilon(j)}-P_{k_j})\rangle_Fdsdt,
\end{align*}
where $\bar \Delta_{i,j}= \Delta_i+\mathcal{R}(P_i)-\mathcal{R}(P_{k_j})$.
Note that $\lim_{j\rightarrow \infty}|P_{L_\varepsilon(j)}-P_{k_j}|=\varepsilon$, as $\lim_{k\rightarrow \infty}h_k=0$.
Then, since $P_{k_j}$ is bounded,
\begin{align*}
\lim_{j\rightarrow \infty}\left|\int_0^1\int_0^1\langle\frac{d}{ds}\partial_x\mathcal{V}(P_{k_j}+st(P_{L_\varepsilon(j)}-P_{k_j})),(P_{L_\varepsilon(j)}-P_{k_j})\rangle_Fdsdt \right|= O(\varepsilon^2).
\end{align*}
Since $\lim_{j\rightarrow \infty}\sum_{i=k_j}^{L_\varepsilon(j)-1} h_i W_i=0$, there exists a sufficiently large $\bar j$, such that for all $j>\bar j$, by choosing sufficiently small $\varepsilon$ and $\delta_0$, we have 
$|\bar \Delta_{i,j}|<{\highlight \varepsilon_\delta}$, and by \cref{thm5_pf_equ1} and \cref{thm5_pf_equ2} it follows that
\begin{align*}
\mathcal{V}(P_{L_\varepsilon(j)})-\mathcal{V}(P_{k_j})
&\leq\langle\partial_x\mathcal{V}(P_{k_j}),\sum_{i=k_j}^{L_\varepsilon(j)-1}h_iW_i\rangle_F-\delta\sum_{i=k_j}^{L_\varepsilon(j)-1}h_i +O(\varepsilon^2)\\
&\leq \langle\partial_x\mathcal{V}(P_{k_j}),\sum_{i=k_j}^{L_\varepsilon(j)-1}h_iW_i\rangle_F-{\highlight \delta\varepsilon/{\highlight \varepsilon_C}}+O(\varepsilon^2)<0.
\end{align*}
This implies that for a large enough $k$, if $P_k\in\{P_{k_j}\}$, then there exists $k'>k$, such that $ \mathcal{V}(P_{k'})<C_2$, and $P_i$ stays in a $\varepsilon$-neighborhood of $P_k$ for $k\leq i\leq k'$.
{\highlight Thus, $P_k$ is bounded, and the proof of part (i) is concluded by contradiction.}

Now, we prove part (ii).
First, rewrite the updating equation in \cref{alg1} as
\begin{align*}
P_{k+1}= P_k&+h_k(\mathcal{R}(P_k)+\Delta_k+W_k)+Z_k,\qquad   k\geq N,\ P_k\in\mathcal{S}_0,
\end{align*}
where  $N$ is chosen as in part (i), and {\mymark the projection term $Z_k$} is defined as
\begin{align*}
{\mymark Z_k=
\begin{cases}
P_0- P_{k+1/2},& \text{if $ P_{k+1/2}\not\in \mathcal{S}_0$},\\
0,&\text{otherwise}.
\end{cases}}
\end{align*}
Define the following continuous-time interpolation:
\begin{align*}
P^0(t)=
\begin{cases}
P_0, & t\leq0,\\
P_k,&  t\in[t_k,t_{k+1}),
\end{cases}
\quad
\Delta^0(t)=
\begin{cases}
\Delta_0, & t\leq0,\\
\Delta_k,&  t\in[t_k,t_{k+1}),
\end{cases}
\end{align*}
where  $t_0=0$ and $t_k=\sum_{i=0}^{k-1}h_i$, for $k\geq 1$.
Define the shifted process $P^k(t)=P^0(t_k+t)$ and $\Delta^k(t)=\Delta^0(t_k+t)$, for all $ t\in\mathbb{R}$.

Then, we have for all $k\geq N$ and $t\geq0$ that
\begin{align}
P^k(t)&=P_k+\sum_{i=k}^{m(t+t_k)-1}h_i(\mathcal{R}(P_i)+\Delta_i)+W^k(t)+Z^k(t)\notag\\
&=P_k+H^k(t)+e^k(t)+W^k(t)+Z^k(t),\label{pk_equ1}
\end{align}
where 
\begin{align*}
H^k(t)&  =\int_0^t   (\mathcal{R}(P^k (s))+ \Delta^k(s))ds,\quad Z^k(t) =  \sum_{i=k}^{m(t+t_k)-1}  Z_i,\\
   W^k(t) &=  \sum_{i=k}^{m(t+t_k)-1}  h_iW_i,\quad
 m(t)=
\begin{cases}
j,& 0\leq t_j\leq t< t_{j+1},\\
0,&t<0,
\end{cases}
\end{align*}
$e^k(t)$ is due to replacing the second  term on the right-hand side of the first equality in \cref{pk_equ1} with $H^k(t)$.
By convention, the above definition assumes $\sum_{i=k}^{m(t+t_k)-1}*=0$, when $0\leq t<h_k$.
Note that for all $\omega\in E$, $W^k(\cdot,\omega)$ converges to $0$ uniformly on any finite time interval.

Fixing  $T>0$ and following the proof of \cite[Theorem 3.3]{Bian2016b}, we can show that  $\{H^k(\cdot)\}_{k=N}^\infty$, $\{Z^k(\cdot)\}_{k=N}^\infty$, and  $\{e^k(\cdot)\}_{k=N}^\infty$ are all relatively compact in $\mathcal{D}([0,T],\mathcal{S}^{n})$, where $\mathcal{D}([0,T],\mathcal{S}^{n})$ denotes the space of functions from $[0,T]$ to $\mathcal{S}^{n}$, that are right-continuous with left-hand limits, equipped with the Skorokhod topology \cite{Skorokhod1956}.
Following the procedure in the proof of \cite[Lemma 3.4]{Abounadi2002}, one can show that the limit of $\{Z^k(\cdot)\}_{k=N}^\infty$ is identically $0$.
Then, the limit of $\{P_k,\Delta_k\}$ satisfies
\begin{align*}
\dot P= \mathcal{R}(P)+\Delta,
\end{align*}
where $\Delta$ converges to $0$ by its definition.
By part (i), we know $\{P_k\}_{k=N}^\infty$ remains in the region of attraction of $P^*$.
Thus, part (ii) is established by \cref{thm2}, part (ii) and the Part 2 of the proof of \cite[Theorem 5.2.1]{Kushner2003}.


{\highlight
To prove part (iii), we note from the part (iii) of \cref{thm2} that the following coupled system is  asymptotically stable at $(P^*,M^*)$:
\begin{align*}
\dot P= \mathcal{R}(P)+\Delta(P,M),\\
\dot M = f(M,P).
\end{align*}
Moreover, by defining $\bar {\mathcal{V}}(P, M) = \bar V(P-P^*, M-M^*)$, where the Lyapunov function $\bar V$ is defined in the proof of \cref{thm2}, we also have
\begin{align*}
\langle\partial_{x_1} \bar {\mathcal{V}}( P, M),(\mathcal{R}(P)+\Delta+\zeta)\rangle_F+\langle\partial_{x_2} \bar {\mathcal{V}}( P, M),f(M,P)\rangle_F<0,
\end{align*}
for all $(P, M)$ in a small neighborhood of $(P^*,M^*)$ with $( P, M)\neq(P^*,M^*)$.
Since $M_k$ is bounded, $\Delta_k$ is bounded for all bounded $P_k$.
Now, following the steps in part (i), we have $(P_k,M_k)$ is bounded, provided $P_0$ stays in a small neighborhood of $P^*$, and $\varepsilon_0$ is small enough.
Applying the analysis in part (ii), we know $(P_k,M_k)$ converges to the solution to the above coupled ODE. 
By the part (iii) of \cref{thm2}, this completes the proof.
}
\end{proof}



\begin{remark}
The first two parts of \cref{thm5} focus on handling static uncertainties represented by either a bounded external disturbance input or a bounded function of $P_k$.
The third part of \cref{thm5} deals with dynamic uncertainty, and hence is more suitable for developing decentralized VI algorithms.
\end{remark}

The following corollary is a direct extension of  \cref{col1}, part (i) and \cref{thm5}, part (iii), and thus its proof is omitted:
 \begin{corollary}\label{col2}
Given $Q_0=Q_0^T>0$, $R_0=R_0^T>0$, and $\lambda>0$, denote $Q=\lambda Q_0$ and $R=o(\lambda) R_0$. 
Suppose $B$ has full rank.
{\highlight
For any $\gamma>0$, if the conditions in the part (iii) of \cref{thm2} are satisfied, then there exist $\lambda>0$, such that $\lim_{k\rightarrow\infty}(P_k,M_k)=(P^*,M^*)$ locally, where $M_k$ is defined in \cref{thm5_equ1}. }
 \end{corollary}
 


\begin{remark}
The boundedness of $M_k$  can be relaxed, by extending the projection term in \cref{thm5_equ1} to the adaptive boundary case as in \cref{alg1}.
The conclusions of \cref{thm5} and \cref{col2} still hold, under minor changes of the proof.
\end{remark}


The following corollary plays an important role in developing adaptive optimal control methods on the basis of the proposed robust VI framework.

\begin{corollary}\label{col3}
%
Denote a complete probability space $(\Omega,\mathcal{F}, \mathbb{P})$ equipped with a filtration $\{\mathcal{F}_k\}_{k\in\mathbb{Z}_+}$.
Consider \cref{alg1} with $W_k=\sigma(P_k)v_k$, $\Delta_k=\Delta_k(P_k)$, and $\sum_{k=0}^\infty h_k^2<\infty$, where $\lim_{k\rightarrow\infty}\Delta_k=0$ uniformly on any compact set, $\sigma_i$ are continuous, and $v_k$ is a $\mathcal{F}_k$-adapted martingale difference with finite variance.
Then, $\lim_{k\rightarrow\infty}P_k=P^*$  with probability one.
\end{corollary}
\begin{proof}
We only need to show that $P_k$ is bounded.
Then, the convergence is proved by the part (ii) of \cref{thm5}.

Again, by contradiction, suppose $\{P_k\}_{k=0}^\infty$ is unbounded.
Following the analysis in the proof of \cref{thm5}, part (i), we  still have
\begin{align*}
\varepsilon<|P_{L_\varepsilon(j)}-P_{k_j}|&=\left|\sum_{i=k_j}^{L_\varepsilon(j)-1}h_i(\mathcal{R}(P_i)+\Delta_i(P_i)+\sigma_i(P_i)v_i)\right|\leq {\highlight \varepsilon_C}\sum_{i=k_j}^{L_\varepsilon(j)-1}h_i,
\end{align*} 
for some ${\highlight \varepsilon_C}>0$, where $\varepsilon$, $k_j$, and $L_\varepsilon(j)$ follow the same definitions in the proof of part (i) of \cref{thm5}.
Since $\lim_{k\rightarrow\infty}\Delta_k=0$  uniformly on any compact set, $\sup_{i\in[k_j, L_\varepsilon(j)]}|\Delta_i(P_i)|$ can be made arbitrarily small, by choosing a large enough $j$.
Then, there exists a sufficiently large $\bar j$, such that for all $j>\bar j$, 
\begin{align}
V(P_{L_\varepsilon(j)})-V(P_{k_j})\leq\langle\partial_xV(P_{k_j}),\sum_{i=k_j}^{L_\varepsilon(j)-1} h_i \sigma_i(P_i)v_i\rangle_F-\delta\varepsilon/{\highlight \varepsilon_C} +O(\varepsilon^2).\label{col3_pf_equ1}
\end{align}

Now, define a sequence $\{M_k\}$, such that 
\begin{align*}
M_k=\sum_{i\in\cup_{j\in\{j\in\mathbb{Z}_+:k_j'\leq k\}}[k_j,L_\varepsilon(j)-1]\cap \mathbb{Z}_+} h_i \sigma_i(P_i)v_i,
\end{align*}
where $k_j'$ is defined in \cref{thm5_pf_equ4}.
Obviously, $\{M_k\}$ is a martingale with respect to $\{\mathcal{F}_k\}$, and $\mathbb{E}(|M_k|^2)$ is bounded, since $P_i$ is bounded,  $\sum_{k=0}^\infty h_k^2<\infty$,  and $v_i$ has finite variance.
By the martingale convergence theorem \cite[Theorem 2.6]{Steele2001}, $M_k$ converges with probability one, and thus $\lim_{j\rightarrow \infty}\sum_{i=k_j}^{L_\varepsilon(j)-1} h_i \sigma_i(P_i)v_i=0$.
This, together with \cref{col3_pf_equ1}, shows that $P_k$ is bounded with probability one.
\end{proof}


\section{ Applications to adaptive/stochastic/decentralized optimal control}\label{sec_applicaiton}
In this section, we provide four applications of the above robust VI method in solving adaptive optimal control problems that appear intractable  using traditional DP methods.

\subsection{VI in the presence of modeling errors}\label{modelerr}
Solving optimal control problems using \cref{alg1} requires precise knowledge of system matrices.
In practice, these system parameters may not be directly available, and are often estimated from measurement data subject to stochastic noise.
In this subsection, we investigate the convergence of \cref{dmre_equ1} under estimated model parameters.


Suppose the system matrix $A$ is not known {\it a priori}, and is approximated by a time series $\{\hat A(t)\}_{t\in\mathbb{R}_+}$, where $\hat A(t)=A+\sum_{i=1}^N\Delta_iv_i(t)$, $N>0$, $\Delta_i\in\mathbb{R}^{n\times n}$ are constants, and $v_i$ denote independent continuous-time Gaussian white noises.

Instead of \cref{dmre_equ1}, let us consider the following equation:
\begin{align}
 \dot P&=\hat A^TP+P\hat A-PBR^{-1}B^TP+Q\notag\\
 &=\mathcal{R}(P)+\sum_{i=1}^N(\Delta_i^TP+P\Delta_i)v_i\notag\\
  &=\mathcal{R}(P)+\sum_{i=1}^N(\Delta_i^T\tilde P+\tilde P\Delta_i)v_i+\sum_{i=1}^N(\Delta_i^T P^*+ P^*\Delta_i)v_i.\label{dmre_equ3}
\end{align}
Since $\Delta_i$ are constants, there exists a constant $\gamma>0$, such that $\sum_{i=1}^N|\Delta_i^TP+P\Delta_i|^2<\gamma|P|^2$.
By \cref{thm2}, part (iv), if $\gamma$ is sufficiently small, then we can find a pair $(Q,R)$ under which there exists a smooth Lyapunov function $V$ satisfying 
   \begin{align*}
\mathcal{L} V(\tilde P)&\leq-C_1 |\tilde P|^2 +\gamma C_2|\tilde P|^2+ \gamma  C_2| P^*|^2\\
&\leq-C_3 |\tilde P|^2 + \gamma  C_2| P^*|^2,\quad \forall P\in\{P\in\mathcal{S}_+^n: P\in B_\varepsilon(P^*)\},
 \end{align*}
for some constants $\varepsilon>0$ and $C_i>0$, $i=1,2,3$.
Moreover, the noise-to-state stability gain \cite {Krstic1998} can be made arbitrarily small by choosing $Q$ and $R$ properly, if $B$ has full rank.

The above inequality shows that the DMRE under noisy measurement of $A$ is either locally or semi-globally (in $\mathcal{S}^n_+$) practically stable, with probability one.
Indeed, due to the additive noise in \cref{dmre_equ3}, $\lim_{t\rightarrow \infty}P(t)$ follows a steady state distribution.

To improve the convergence result, lets consider the following DMRE:
\begin{align}
 \dot P=\left(\frac{1}{t}\int_0^t\hat Ads\right)^TP+P\left(\frac{1}{t}\int_0^t\hat Ads\right)-PBR^{-1}B^TP+Q.\label{dmre_equ4}
\end{align}
By definition,
\begin{align*}
\frac{1}{t}\int_0^t\hat Ads=A+\sum_{i=1}^N\Delta_i\frac{1}{t}\int_0^tv_i(s)ds=A+\sum_{i=1}^N\frac{1}{t}\Delta_iw_i(t),
\end{align*}
where $w_i$ are independent Brownian motions \cite[Chapter 3]{Steele2001}, and the last equality comes from the fact that $v_i$ are independent Gaussian white noises.
By the strong law of large number \cite[Appendix I]{Steele2001}, $\lim_{t\rightarrow\infty}\frac{1}{t}w_i(t)=0$ with probability one.

\begin{theorem}
Denote a complete probability space $(\Omega,\mathcal{F}, \mathbb{P})$ equipped with a filtration $\{\mathcal{F}_t\}_{t\geq0}$.
Suppose $(w_1(t),w_2(t),\cdots,w_N(t))$ is $\mathcal{F}_t$-adapted.
For any $P(0)\in\mathcal{S}^n_+$,  we have $\lim_{t\rightarrow \infty}P(t)=P^*$ with probability one, where $P(t)$ is defined by \cref{dmre_equ4}. 
\end{theorem}
\begin{proof}
By the definition of $w_i$, we know there exists $E\in\mathcal{F}$ with $\mathbb{P}(E)=1$, such that for any $\omega\in E$, $\lim_{t\rightarrow\infty}\frac{1}{t}\sum_{i=1}^N\Delta_iw_i(\omega,t)=0$.
Now, fix a $\omega\in E$, and denote $\frac{1}{t}\sum_{i=1}^N\Delta_iw_i(\omega,t):=\Delta(t)$.

Define 
\begin{align*}
\dot {\hat P} =(A-BK^*)^T\hat P+\hat P(A-BK^*)+K^*RK^*+Q+\Delta^T\hat P+\hat P\Delta,\quad \hat P(0)=P(0).
\end{align*}
Since $\lim_{t\rightarrow\infty}\Delta(t)=0$, we can easily show that $\hat P$ is globally asymptotically stable at $P^*$ on $\mathcal{S}^n$, by using the same Lyapunov function given in the proof of \cref{thm1}.
On the other hand, we have from \cref{dmre_equ4} that
\begin{align*}
0\leq x^T(-t)P(t)x(-t)=\inf_u\left\{x^T(0)P(0)x(0)+\int_{-t}^{0}(x^T Qx+u^TRu)ds\right\}\notag\\
\leq (x^*(0))^TP(0)x^*(0)+\int_{-t}^{0}(x^*)^T (Q+(K^*)^TRK^*)x^*ds=x^T(-t)\hat P(t)x(-t),
\end{align*}
where $x$ is the solution to the following system:
\begin{align*}
\dot x = (A+\Delta)x+Bu.
\end{align*}
Thus, $P(t)$ is bounded and stays in the region of attraction of $P^*$, for each given $\omega\in E$.
The proof is then concluded following the similar analysis in the proof of \cref{thm2}, part (ii).
\end{proof}


In practice, instead of having $v_i(t)$, we usually have discrete-time white noise sequences $\{v_i(k)\}_{k=0}^\infty$ sampled from the continuous-time series, with constant variance $\sigma_i^2$.
In this case, \cref{dmre_equ3} and \cref{dmre_equ4} can be numerically approximated by
\begin{align*}
  P_{k+1}&=P_k+h_k(\hat A_k^TP_k+P_k\hat A_k-P_kBR^{-1}B^TP_k+Q),\\
    P_{k+1}&=P_k+h_k\left(\left(\frac{1}{k}\sum_{i=0}^k\hat A_i\right)^TP_k+P_k\left(\frac{1}{k}\sum_{i=0}^k\hat A_i\right)-P_kBR^{-1}B^TP_k+Q\right),
\end{align*}
respectively, where $ P_0=P(0)$, $\hat A_k=A+\sum_{i=1}^N\Delta_iv_i(k)$, and $h_k>0$ is the step size.
By the property of Gaussian white noise, we have $\lim_{h\rightarrow 0}\sum_{k=0}^{\lfloor t/h\rfloor}\sqrt{h} v_i(k)=\sigma_i w_i(t)$.
Then the convergence of $P_k$ can also be obtained.

\begin{remark}
Note that the system input matrix $B$ can also be replaced by a time series $\hat B(t)$ in the above analysis.
\end{remark}

\subsection{VI-based ADP for linear continuous-time systems}\label{VIADP}


ADP aims at solving the optimal control problem in real-time using the online input-state or input-output information.
However, in traditional PI-based ADP methodologies, a matrix inverse is calculated in each learning iteration, which may induce a heavy computational burden in real world applications.
Here, we solve this problem from the perspective of robust DP.
This result is especially useful for high-order systems where solving matrix inverse online is not practical.


For all $x\in\mathbb{R}^n$ and $P\in\mathcal{S}^n$, taking the derivative along the solutions of system \cref{sys_equ1}, one has
\begin{align}
\frac{d}{dt}\bar x^T \vecs(P)&=\frac{d}{dt}(x^T Px)=(Ax+Bu)^TPx+x^TP(Ax+Bu)
=\bar z^T \theta(P),\label{VIADP_equ1}
\end{align}
where $z=[x^T,u^T]^T$, 
\begin{align*}
\theta(P)=\vecs\left(
\begin{bmatrix}
PA+A^TP&PB\\
B^TP&0\\
\end{bmatrix}\right),
\end{align*}
and for any $\xi\in \mathbb{R}^q$,  $q\in\mathbb{Z}_+\setminus\{0\}$, 
\begin{align*}
\bar \xi=[\xi_1^2,2\xi_1\xi_2,\cdots,2\xi_1\xi_q,\xi_2^2,2\xi_2\xi_3,\cdots,2\xi_{q-1}\xi_q,\xi_q^2]^T.
\end{align*} 
Note that once $\theta(P)$ is obtained, we can define two linear transformations $\mathcal{T}_A$ and $\mathcal{T}_B$, such that 
\begin{align*}
A^TP+PA=\mathcal{T}_A(\theta(P)),\quad R^{-1}B^TP=\mathcal{T}_B(\theta(P)).
\end{align*}

To provide an online implementation of \cref{alg0}, we need to solve $A^TP+PA$ and $R^{-1}B^TP$ from \cref{VIADP_equ1} using online data only.
First, define an arbitrary time sequence $0\leq t_1<t_2<\cdots<t_{l+1}<\infty$.
Consider the following linear equation:
\begin{align}
\psi_j^T(z)\theta(P)= \phi_j^T(x) \vecs(P),\quad \forall j\geq0,\label{adp_equ4}  
\end{align}
where 
\begin{align*}
\phi_j(x)=\bar x(t_{j+1})-\bar x(t_j),\quad \psi_j(z)=\int_{t_j}^{t_{j+1}}\bar zdt.
\end{align*}

Now, by means of the RLS \cite[Chapter 10]{Haykin2014}, one can define a sequence $\{\theta_k\}_{k=0}^l$ to approximate $\theta(P)$.
To be specific, $\theta_k$ is updated by the following two equations:
\begin{align}
\Sigma_k&=\Sigma_{k-1}-\frac{\Sigma_{k-1}\psi_k\psi_k^T\Sigma_{k-1}}{1+\psi_k^T\Sigma_{k-1}\psi_k},\notag\\
\theta_k&=\theta_{k-1}+\Sigma_k\psi_k\phi_k^T\vecs(P)-\Sigma_k\psi_k\psi_k^T\theta_{k-1},\label{RLS_equ2}
\end{align}
where $\theta_0=0$, and $\Sigma_0=\lambda^{-1}I_q$ for $q=\frac{1}{2}((n+m)^2+n+m)$ and some $\lambda>0$.

\begin{assumption}\label{assum1}
There exist $l_0>0$ and $\alpha>0$, such that 
\begin{align}
 \frac{1}{l}\sum_{j=1}^l\psi_{j}(z)\psi_{j}^T(z)>\alpha I
\end{align}
for all $l>l_0$.
\end{assumption}

If \cref{assum1} is satisfied, then we have from \cref{RLS_equ2} that
 \begin{align*}
\theta_l&=\left(\sum_{j=1}^l\psi_j\psi_j^T+\lambda I_q\right)^{-1}\sum_{j=1}^l\psi_j\phi_j^T\vecs(P)\\
&=\left(\frac{1}{l}\sum_{j=1}^l\psi_j\psi_j^T+\frac{\lambda}{l} I_q\right)^{-1}\frac{1}{l}\sum_{j=1}^l\psi_j\phi_j^T\vecs(P),
\end{align*}
and thus
 \begin{align*}
 \lim_{l\rightarrow \infty}\theta_l=\theta(P).
\end{align*}
By \cref{RLS_equ2} and using mathematical induction,  we see $\theta_k$ is also linear in $P$.
Hence one can find a matrix $M_k\in\mathbb{R}^{\left(\frac{n(n+1)}{2}+mn\right)\times \frac{n(n+1)}{2}}$, such that $\theta_k=M_k\vecs(P)$, with $M_0=0$, for all $k=0,1,\cdots,l$.
Moreover, since \cref{VIADP_equ1} is true for any $P\in\mathcal{S}^n$, by replacing $\theta_k$ in \cref{RLS_equ2} with $M_k\vecs(P)$, we must have
\begin{align}
M_k&=M_{k-1}+\Sigma_k\psi_k\phi_k^T-\Sigma_k\psi_k\psi_k^TM_{k-1}.\label{RLS_equ3}
\end{align}
By the convergence  of $\theta_l$, we have
 \begin{align*}
 \lim_{l\rightarrow \infty}M_l= \lim_{k\rightarrow \infty}M_k=M,
\end{align*}
where $M$ satisfies $\theta(P)=M\vecs(P)$ for all $P\in\mathcal{S}^n$.

Based on the above analysis, the VI-based ADP algorithm for linear continuous-time systems is given in \cref{alg2}.


\begin{algorithm}[t]
\caption{Continuous-time VI-based ADP}
\label{alg2}
\begin{algorithmic}
\STATE Choose $P_0=P_0^T\geq0$, $\Sigma_0=\lambda^{-1}I$, and $M_0=0$. $k, q\gets 0$.
\STATE Apply a measurable locally essentially bounded input $u$ to \cref{sys_equ1}.
\LOOP
    \STATE $\Sigma_k\gets \Sigma_{k-1}-\Sigma_{k-1}\psi_k\psi_k^T\Sigma_{k-1}/(1+\psi_k^T\Sigma_{k-1}\psi_k)$
    \STATE $M_k\gets M_{k-1}+\Sigma_k\psi_k(\phi_k^T-\psi_k^TM_{k-1})$
    \STATE $\hat \theta_k\gets M_k\vecs(P_k)$
    \STATE $ P_{k+1/2}\gets P_k+h_k(\mathcal{T}_A(\hat \theta_k)-\mathcal{T}_B^T(\hat \theta_k)R\mathcal{T}_B(\hat \theta_k)+Q)$
    \IF {$P_{k+1/2}>0$ and $| P_{k+1/2}-P_k|/h_k<\bar\varepsilon$}
     \RETURN $P_{k}$ as an approximation to $P^*$
    \ELSIF {$| P_{k+1/2}|>B_q$ or $ P_{k+1/2}\not>0$}
    \STATE $P_{k+1}\gets P_0$. $q\gets q+1$.
    \ELSE
        \STATE $ P_{k+1}\gets  P_{k+1/2}$
    \ENDIF
    \STATE $k\gets k+1$
 \ENDLOOP
 \end{algorithmic}
\end{algorithm}



\begin{theorem}\label{thm6}
Under \cref{assum1}, we have $\underset{k\rightarrow \infty}{\lim}P_k=P^*$ and $\underset{k\rightarrow \infty}{\lim}K_k=K^*$, where $\{P_k\}_{k=0}^\infty$ is obtained from \cref{alg2}, and $K_k=\mathcal{T}_B(\hat \theta_k)$.
\end{theorem}

\begin{proof}
Noting that 
 \begin{align*}
 \theta(P_k)
 =M\vecs(P_k)=\hat \theta_k+(M-M_k)\vecs(P_k),
\end{align*}
we have
 \begin{align*}
\mathcal{T}_A(\hat \theta_k)&=A^TP_k+P_kA+\Delta_{1,k}(P_k),\\
\mathcal{T}_B(\hat \theta_k)&=R^{-1}B^TP_k+\Delta_{2,k}(P_k),
\end{align*}
for some  linear functions $\Delta_{1,k}$ and $\Delta_{2,k}$.
Thus, \cref{alg2} is essentially a special case of \cref{alg1} with
 \begin{align*}
\Delta_k&=\Delta_{1,k}(P_k)+\Delta_{2,k}^T(P_k)R\Delta_{2,k}(P_k)+\Delta_{2,k}^T(P_k)B^TP_k+P_kB\Delta_{2,k}(P_k).
\end{align*}
Since $\lim_{k\rightarrow \infty}M_k=M$ under \cref{assum1}, both $\Delta_{1,k}$ and $\Delta_{2,k}$ converge to  zero over any compact set.
Then, the convergence of $P_k$ to $P^*$ is proved by \cref{thm5}, part (ii). 
Following the definition, we then easily have $\lim_{k\rightarrow \infty}K_k=K^*$.
 This completes the proof.
\end{proof}

\begin{remark}
Since the RLS scheme is also robust to stochastic noise \cite[Chapter 10]{Haykin2014},  it is possible to extend \cref{alg2} to the stochastic optimal control framework.
\end{remark}

\subsection{Stochastic ADP for ergodic control problems}\label{stoADP}

In this subsection, we develop an ADP algorithm to solve the ergodic control problem \cite{Borkar2006} for linear stochastic systems with additive noise.

Consider the following system:
\begin{align}
dx&=Axdt+Budt+\sum_{i=1}^{q_x}\sigma_idw_{x,i},\label{sys_equ4}\\
du&=-K_0dx+\sum_{i=1}^{q_u}\sigma_{u,i}dw_{u,i},\label{control_equ4}
\end{align}
where 
$x$, $u$, $A$, and $B$ follow the same definitions as in system \cref{sys_equ1};
$x(0)$ is deterministic;
$w_{x,i}$ and $w_{u,i}$ are independent Brownian motions;  
$q_x,q_u\in\mathbb{Z}_+$;
$\sigma_i\in \mathbb{R}^{n}$ are unknown constant vectors;
$K_0$ is a known initial input matrix; and
$\sigma_{u,i}\in \mathbb{R}^{m}$ are constant vectors.

\begin{remark}\label{rem4_3}
$\sum_{i=1}^{q_x}\sigma_idw_{x,i}$ in \cref{sys_equ4} represents the additive noise in  system \cref{sys_equ4}.
$\sum_{i=1}^{q_u}\sigma_{u,i}dw_{u,i}$ in \cref{control_equ4} serves as an exploration noise, which has been widely used in adaptive control literature to guarantee the persistent excitation condition (PE) \cite[Definition 3.2]{Tao2003}.
Note that besides the Brownian motion, other types of exploration noises can also be used.
For simplicity, we only consider inputs in the form of \cref{control_equ4} here, as in this case system \cref{sys_equ4} is purely driven by Brownian motions, and several standard results from SDE theory can be applied directly.
\end{remark}

\begin{assumption}\label{assum2}
There exists an ergodic stationary probability measure $\mu$ on $\mathbb{R}^n\times \mathbb{R}^m$ for system \cref{sys_equ4}-\cref{control_equ4}.
\end{assumption}

{\highlight A discrete time version  of \cref{assum2} for MDPs has been widely used in approximate DP and RL literature \cite{Tsitsiklis1994,Tsitsiklis1997}.
See \cite{Wonham1967,Haussmann1971} for conditions under which \cref{assum2} holds.}

The objective of ergodic control is to minimize (with probability one)
\begin{align*}
\mathcal{J}(u)=\limsup_{T\rightarrow \infty}\frac{1}{T}\int_0^{T}(x^TQx+u^TRu)dt,
\end{align*}
where $Q=Q^T>0$ and $R=R^T>0$.
It can be shown \cite{Borkar2006} that $\inf_u\mathcal{J}(u)=\sum_{i=1}^{q_3}\sigma_i^TP^*\sigma_i$, with $P^*$ and the optimal controller sharing the same definitions as the ones in \Cref{MP_SD} for deterministic systems.

Now, we derive an online ADP algorithm to solve the above ergodic control problem.
Similar to \cref{VIADP_equ1}, for all $x\in\mathbb{R}^n$ and $P\in\mathcal{S}^n$, by It{\^o}'s lemma \cite[Theorem 8.3]{Steele2001}, we have along the trajectories of \cref{sys_equ4} that 
\begin{align}
{d}(x^T Px)
&=2x^TP(Ax+Bu)dt+\sum_{i=1}^{q_x}\sigma_i^TP\sigma_idt+2x^TP\sum_{i=1}^{q_x}\sigma_idw_{x,i}\notag\\
&= \psi^T(x,u) \theta(P)dt+2x^TP\sum_{i=1}^{q_x}\sigma_idw_{x,i},\label{VIADP_equ4}
\end{align}
where $\psi(x,u)=[\bar x^T,x^T\otimes u^T,1]^T$, and
\begin{align*}
\theta(P)=
\begin{bmatrix}
\vecs(PA+A^TP)\\
\ves(B^TP)\\
\sum_{i=1}^{q_x}\sigma_i^TP\sigma_i
\end{bmatrix}.
\end{align*}
Then, multiplying $ \psi$ on both sides of \cref{VIADP_equ4}, we have on any finite time interval $[0,T]$ that 
\begin{align}
\frac{1}{T}\int_0^T \psi\psi^Tdt\theta(P)&=\frac{1}{T}\int_0^T\psi d(x^TPx)-\frac{2}{T}\int_0^T\psi x^TP\sum_{i=1}^{q_x}\sigma_idw_{x,i}.
\label{stoVIadp_equ1}
\end{align}
Once $\theta(P)$ is obtained from the equation above, we can use the linear transformations defined similarly as the ones in \Cref{VIADP} to find $A^TP+PA$ and $R^{-1}B^TP$:
\begin{align*}
A^TP+PA=\mathcal{T}_A(\theta(P)),\quad R^{-1}B^TP=\mathcal{T}_B(\theta(P)).
\end{align*}




{\highlight
In order to solve \cref{stoVIadp_equ1}, we impose the following assumption:
\begin{assumption}\label{assum3}
$\mu$ satisfies
\begin{align}
\int_{\mathbb{R}^n\times \mathbb{R}^m}\psi\psi^Td\mu>0.\label{equ_PE2}
\end{align}
\end{assumption}
Note that system \cref{sys_equ4}-\cref{control_equ4} is a multidimensional Ornstein-Uhlenbeck process.
Then, its stationary probability measure $\mu$ is also Gaussian, and thus $(x,u)$ has finite $r$-th moment for any $r\in\mathbb{N}_+$.
\cref{assum3} is similar to the PE condition widely used in adaptive control literature (see \cref{rem4_3} for details).}

By a direct extension of Birkhoff's ergodic theorem \cite[Theorem 1.5.18]{Arapostathis2012} and the It{\^o}'s isometry \cite[Theorem 6.1]{Steele2001}, we know\footnote{\highlight
$\|\cdot\|_2$ denotes the matrix 2-norm.
$\mathbb{E}^{\mathbb{P}}$ is the expectation on a probability space $(\Omega,\mathcal{F},\mathbb{P})$, where $\Omega$ is a sample space, $\mathcal{F}$ is a $\sigma$-field of Borel sets of $\Omega$, and $\mathbb{P}$ is a stationary distribution of $(x,u)$ such that $\int_\Omega f(x(\omega),u(\omega))d\mathbb{P}(\omega)=\int_{\mathbb{R}^n\times\mathbb{R}^m} f(x,u)d\mu(x,u)$ for all measurable $f$.}
{\highlight
\begin{align}
\lim_{T\rightarrow\infty}\mathbb{E}^{\mathbb{P}}\left[\left\|\frac{1}{T}\int_0^T\psi\psi^Tdt-\int_{\mathbb{R}^n\times \mathbb{R}^m}\psi\psi^Td\mu\right\|_2^2\right]&=0,\label{ergo_equ1}\\
\lim_{T\rightarrow\infty}\mathbb{E}^{\mathbb{P}}\left[\left\|\frac{1}{T}\int_0^T\psi d(x^TPx)-\frac{1}{T}\int_0^T\psi\psi^Tdt\theta(P)\right\|_2^2\right]&\notag\\
=\lim_{T\rightarrow\infty}\frac{4}{T^2}\mathbb{E}^{\mathbb{P}}\left[\left\|\int_0^T\psi x^TP\sum_{i=1}^{q_x}\sigma_idw_{x,i}\right\|_2^2\right]&=0
\label{ergo_equ2}.
\end{align}}

Choosing a monotone increasing sequence $\{t_k\}_{k=0}^\infty$ with $t_0>0$ and $\lim_{k\rightarrow\infty}t_k=\infty$, we denote 
\begin{align*}
\hat \theta(P,t_k)=\left(\int_0^{t_k}\psi\psi^Tdt\right)^{-1}\int_0^{t_k}\psi d(x^TPx).
\end{align*}
Note from \cref{equ_PE2} and \cref{ergo_equ1} that $t_0$ is a stopping time, and $t_0<\infty$ with probability one, such that $\int_0^{t_k}\psi\psi^Tdt$ is invertible for all $k\geq0$.

For simplicity, denote $\hat \theta_k=\hat \theta(P_k,t_k)$.
The VI-based ADP algorithm for the ergodic control problem is given in \cref{alg5}.


\begin{algorithm}[t]
\caption{Online robust optimal control design for ergodic control}
\label{alg5}
\begin{algorithmic}
\STATE Choose $P_0=P_0^T\geq0$. $k, q\gets 0$. Pick an input $u$ in form of \cref{control_equ4}.
\LOOP
    \STATE $\hat \theta_k=\left(\int_0^{t_k}\psi\psi^Tdt\right)^{-1}\int_0^{t_k}\psi d(x^TP_kx)$.
    \STATE $ P_{k+1/2}\gets P_k+h_k(\mathcal{T}_A(\hat \theta_k)-\mathcal{T}_B^T(\hat \theta_k)R\mathcal{T}_B(\hat \theta_k)+Q)$
    \IF {$P_{k+1/2}>0$ and $| P_{k+1/2}-P_{k}|/h_k<\bar\varepsilon$}
    \RETURN $P_{k}$ as an approximation to $P^*$
    \ELSIF {$| P_{k+1/2}|>B_q$ or $ P_{k+1/2}\not>0$}
    \STATE  $P_{k+1}\gets P_0$. $q\gets q+1$.
    \ELSE
        \STATE $ P_{k+1}\gets  P_{k+1/2}$
    \ENDIF
    \STATE $k\gets k+1$
 \ENDLOOP
 \end{algorithmic}
\end{algorithm}


\begin{theorem}\label{thm7}
Under  \cref{assum2} and \cref{assum3}, we have $\underset{k\rightarrow \infty}{\lim}P_k=P^*$ and $\underset{k\rightarrow \infty}{\lim}K_k=K^*$ with probability one, where $\{P_k\}_{k=0}^\infty$ is obtained from \cref{alg5}, and $K_k=\mathcal{T}_B(\hat \theta_k)$.
\end{theorem}
\begin{proof}
{\highlight
First denote 
\begin{align*}
\Delta_k(P):=\hat \theta(P,t_k)-\theta(P)=2\left(\int_0^{t_k}\psi\psi^Tdt\right)^{-1}\int_0^{t_k}\psi x^TP\sum_{i=1}^{q_x}\sigma_idw_{x,i}.
\end{align*}
Then, by \cref{ergo_equ1} and \cref{ergo_equ2}, we have 
\begin{align*}
\lim_{k\rightarrow \infty}\mathbb{E}^{\mathbb{P}}\left[\|\Delta_k(P)\|_2^2\right]=0.
\end{align*}
Since the above formulation is true for any real symmetric $P$, 
we know for any $P\in\mathcal{S}^n$, $\Delta_k(P)$ is a martingale (element wise), and converges to $0$ as $k$ goes to infinite, in the mean square sense.

By Burkholder-Davis-Gundy inequality \cite[Theorem 1.1]{Burkholder1972}, we have 
\begin{align*}
\mathbb{E}^{\mathbb{P}}\left[|\Delta_k^{i,j}(P)|^4\right]\leq C\mathbb{E}^{\mathbb{P}}\left[[\Delta^{i,j}(P)]_k^2\right]
\end{align*}
for some constant $C>0$, where $\Delta_k^{i,j}$ is the $(i,j)$-th element of $\Delta_k$, and $[\cdot]$ denotes the quadratic variation \cite[Section 8.6]{Steele2001}.
By Birkhoff's ergodic theorem and the fact that $(x,u)$ has finite $r$-th moment for any $r\in\mathbb{N}_+$, $\lim_{k\rightarrow\infty}[\Delta^{i,j}(P)]_k=0$ with probability one. 
Hence $\lim_{k\rightarrow\infty}\mathbb{E}^{\mathbb{P}}\left[[\Delta^{i,j}(P)]_k^2\right]=0$.
This implies that the variance of $\Delta_k^T\Delta_k$ is bounded.
}

Now,  the updating equation in \cref{alg5} is equivalent to
\begin{align*}
 P_{k+1/2}\gets P_k+h_k(\mathcal{R}(P_k)+\Delta_{1,k}(P_k)+\Delta_{2,k}(P_k)),
\end{align*}
where $\Delta_{1,k}(P_k)$ is a zero-mean  stochastic noise with finite variance for each $k$, and $\Delta_{2,k}(\cdot)$ is deterministic  and decreases to $0$ as $k$ goes to infinity.
The proof is then completed by \cref{col3}.
%
\end{proof}


\begin{remark}
It is possible to extend the results in this section to systems with both multiplicative and additive noises:
\begin{align*}
dx=Axdt+Budt+\sum_{i=1}^{q_x}\sigma_idw_{x,i}+\sum_{i=1}^{q_1}F_ixdw_{1,i}+\sum_{i=1}^{q_2}G_iudw_{2,i},
\end{align*}
where $F_i$ and $G_i$ are constant matrices.
In this case, the ARE is given as
\begin{align*}
A^TP^*+P^*A+\sum_{i=1}^{q_1}F_i^TP^*F_i-P^*B\left(R+\sum_{i=1}^{q_2}G_i^TP^*G_i\right)^{-1}B^TP^*+Q=0.
\end{align*}
The convergence of VI in this case is guaranteed using results in \cite[Theorem 3.3]{Bian2016b} and \cite{Ait-Rami2001}.
The robust VI and ADP algorithms similar to  \cref{alg1} and \cref{alg5} can be derived following the analysis given before.
\end{remark}



\subsection{Decentralized VI}\label{decom}
In previous sections, we have studied different types of optimal control problems for continuous-time linear systems.
A common feature in these results is that the optimal controller and value function can be obtained by solving a single ARE.
However, in some applications, including the non-zero-sum differential game and the robust ADP, the optimal solution is solved from a group of cascaded or coupled AREs/HJB equations.
Here, we present a decentralized VI framework for continuous-time linear systems based on the robust VI proposed in \Cref{RVI}.

{\highlight
For simplicity, let us consider a network of two agents, with each agent $i$, $i=1,2$, aiming at solve a linear optimal control problem (see \Cref{MP_SD}) defined by four matrices $(A_i,B_i,Q_i,R_i)$.
Obviously, if $(A_1,B_1,Q_1,R_1)$ and $(A_2,B_2,Q_2,R_2)$ are not dependent on each other, then each agent can solve its own optimal control problem without communicating with the other one.
However, assume now agent $i$'s system information $(A_i,B_i,Q_i,R_i)$ depends on agent $j$'s ($j\neq i$) optimal solution $(P_j^*,K_j^*)$ through a nonlinear relationship $\Delta_i(\cdot)$, and for security reason the two agents cannot exchange their system information $(A_i,B_i,Q_i,R_i)$, $i=1,2$, to each other, then it is no longer a trivial task how to solve $(P_i^*,K_i^*)$ in a decentralized manner.
Reformulating this problem mathematically, we focus on solving  the following two coupled AREs:}
\begin{align*}
0&=A_1^TP_1^*+P_1^*A_1-P_1^*B_1R_1^{-1}B_1^TP_1^*+Q_1+\Delta_1(P_1^*,P_2^*),\\
0&=A_2^TP_2^*+P_2^*A_2-P_2^*B_2R_2^{-1}B_2^TP_2^*+Q_2+\Delta_2(P_2^*,P_1^*),
\end{align*}
where  $(A_i,B_i,Q_i,R_i)\in\mathbb{R}^{n_i\times n_i}\times \mathbb{R}^{n_i\times m_i}\times\mathcal{S}^{n_i}_+\times\mathcal{S}^{m_i}_+$, $\Delta_1=\Delta_1^T$ and $\Delta_2=\Delta_2^T$ are two continuous nonlinear functions.

\begin{assumption}\label{assum4}
There exist four polynomials $\gamma_{i,j}\in{\highlight\mathcal{K}}$, $i,j=1,2$, such that\footnote{A function $\gamma:\mathbb{R}_+\rightarrow\mathbb{R}_+$ is of class ${\highlight\mathcal{K}}$, if it is continuous, strictly increasing, and $\gamma(0)=0$.}
\begin{align*}
|\tilde \Delta_1(P_1,P_2)|&\leq \gamma_{1,1}(|\tilde P_1|)+ \gamma_{1,2}(|\tilde P_2|),\\
|\tilde \Delta_2(P_2,P_1)|&\leq \gamma_{2,2}(|\tilde P_2|)+\gamma_{2,1}(|\tilde P_1|),
\end{align*}
 where $\tilde \Delta_1(P_1,P_2)= \Delta_1(P_1,P_2)-\Delta_1(P_1^*,P_2^*)$, $\tilde \Delta_2(P_2,P_1)= \Delta_2(P_2,P_1)-\Delta_2(P_2^*,P_1^*)$, $\tilde P_1= P_1- P_1^*$, and $\tilde P_2= P_2- P_2^*$.
\end{assumption}

\begin{remark}
\cref{assum4} holds widely in different control problems.
For example, in two-player non-zero-sum differential games, we have $A_1=A_2$ and 
\begin{align*}
\Delta_i(P_i,P_j)=P_jB_jR_j^{-1}R_{ij}R_j^{-1}B_j^TP_j-P_jB_jR_j^{-1}B_j^TP_i-P_iB_jR_j^{-1}B_j^TP_j,
\end{align*}
 where $i\neq j$ and $R_{ij}=R_{ij}^T>0$.
Also, in the robust ADP design for systems with unmatched disturbances \cite[Chapter 5.1.1.2]{Jiang2017}, we have $\Delta_1=0$ and 
\begin{align*}
\Delta_2(P_2,P_1)=P_2R_1^{-1}B_1^TP_1B_1+B_1^TP_1B_1R_1^{-1}P_2.
\end{align*}
Note that $\gamma_{i,j}$ may depend on $P_1^*$ and $P_2^*$.
\end{remark}


The following theorem provides a convergence analysis for the coupled DMREs using small-gain theory.

\begin{theorem}\label{thm8}
Under \cref{assum4},  there exist $\varepsilon>0$ and small enough $|\gamma_{i,j}|$, $i,j=1,2$, such that given $(P_1(0),P_2(0))$ in a $\varepsilon$-neighborhood of $(P_1^*,P_2^*)$, we have $\lim_{t\rightarrow\infty}P_1(t)=P_1^*$ and $\lim_{t\rightarrow\infty}P_2(t)=P_2^*$, where
\begin{align}
\dot P_1&=A_1^TP_1+P_1A_1-P_1B_1R_1^{-1}B_1^TP_1+Q_1+\Delta_1(P_1,P_2),\label{dmre_equ5}\\
\dot P_2&=A_2^TP_2+P_2A_2-P_2B_2R_2^{-1}B_2^TP_2+Q_2+\Delta_2(P_2,P_1).\label{dmre_equ6}
\end{align}
Moreover, if $B_i$ has full rank, the convergence result holds for any $\gamma_{i,j}$ by picking $Q_i$ and $R_i$ properly.
\end{theorem}
\begin{proof}
Following the derivation of \cref{thm2_pf2}, there exist $\varepsilon>0$ and a Lyapunov function $V$, such that 
\begin{align*}
\dot V(\tilde P_1,\tilde P_2)&\leq-C_1 (|\tilde P_1|^2+|\tilde P_2|^2)+C_2|\tilde P_1||\tilde \Delta_1|+C_3|\tilde P_2||\tilde \Delta_2|\\
&\leq-C_1 (|\tilde P_1|^2+|\tilde P_2|^2)
+C_2|\tilde P_1|\sum_{j=1,2}\gamma_{1,j}(|\tilde P_j|)
+C_3|\tilde P_2|\sum_{j=1,2}\gamma_{2,j}(|\tilde P_j|)\\
&\leq-\frac{C_1}{2} (|\tilde P_1|^2+|\tilde P_2|^2)+C_4\sum_{i,j=1,2}\gamma_{i,j}^2(|\tilde P_j|),\quad \forall |\tilde P_1|<\varepsilon,\ |\tilde P_2|<\varepsilon,
 \end{align*}
where  $C_i>0$, $i=1,2,3,4$, are constants.
 Since $\gamma_{i,j}$  are polynomials, the second term on the right-hand side of the above inequality decrease to $0$ at least as fast as the first term.
 Hence, \cref{dmre_equ5} and \cref{dmre_equ6} are  asymptotically stable at $(P_1^*,P_2^*)$, as long as the gain of $\gamma_{i,j}$ is small enough.
 
Moreover, if $B_1$ has full rank, we know from  \cref{thm4} that \cref{dmre_equ5} can have an arbitrarily small linear $L^2$ gain from $\tilde \Delta_1$ to $\tilde P_1$, i.e., $C_2/C_1$ can be made sufficiently small, on any compact sets, by choosing $Q_1$ and $R_1$ properly.
If $|\gamma_{2,2}|$ is sufficiently small, then for any $\varepsilon>0$, we can find $Q_1$ and $R_1$, such that
\begin{align*}
\dot V(\tilde P_1,\tilde P_2)\leq&-C_1 (|\tilde P_1|^2+|\tilde P_2|^2)+C_2|\tilde P_1|\gamma_{1,1}(|\tilde P_1|)+\frac{C_2}{2}|\tilde P_1|^2+\frac{C_2}{2}\gamma_{1,2}^2(|\tilde P_2|)\\
&+\frac{C_1}{2}|\tilde P_2|^2+\frac{C_3^2}{2C_1}\gamma_{2,1}^2(|\tilde P_1|)+C_3|\tilde P_2|\gamma_{2,2}(|\tilde P_2|)\\
\leq&-C_5( |\tilde P_1|^2+ |\tilde P_2|^2),\quad  \forall|\tilde P_1|<\varepsilon,\ |\tilde P_2|<\varepsilon,
 \end{align*}
 for some $C_5>0$.
 This completes the proof.
\end{proof}

\begin{algorithm}[t]
\caption{Decentralized value iteration}
\label{alg6}
\begin{algorithmic}
\STATE For the $i$-th subsystem, choose $P_{i,0}=P_{i,0}^T\geq 0$. $k\gets 0$.
\LOOP
    \STATE $P_{i,k+1}\gets P_{i,k}+h_{i,k}(A_i^TP_{i,k}+P_{i,k}A_i-P_{i,k}B_iR_i^{-1}B_i^TP_{i,k}+Q_i+\Delta_i(P_{i,k},P_{j,k}))$
    \IF {$| P_{i,k+1}-P_{i,k}|/h_{i,k}<\bar\varepsilon$}
    \RETURN $P_{i,k}$ as an approximations to $P_i^*$
    \ENDIF
    \STATE $k\gets k+1$
    \ENDLOOP
 \end{algorithmic}
\end{algorithm}
Based on \cref{thm8}, we develop a coupled VI algorithm in \cref{alg6}.
The convergence of \cref{alg6} is given in the following theorem. 

{\highlight 
\begin{theorem}\label{thm9}
Under \cref{assum4}, suppose $B_1$ and $B_2$ have full rank.
If $\sup_k\{h_{i,k}\}$ is sufficiently small, then given $Q_{i,0}\in\mathcal{S}^{n_i}_+$ and $R_{i,0}\in\mathcal{S}^{m_i}_+$, for any $\varepsilon>0$, there exist $\lambda_i>0$, such that by selecting $Q_i=\lambda_i Q_{i,0}$ and $R_i=o(\lambda_i) R_{i,0}$, we have $\lim_{k\rightarrow \infty}P_{i,k}=P_i^*$, where $\{P_{i,k}\}_{k=0}^\infty$ is obtained from \cref{alg6} with $P_{i,0}\in \mathcal{S}^{n_i}_+\cap B_{\varepsilon}(P_i^*)$, and $i=1,2$.
\end{theorem}}
\begin{proof}
{\highlight 
First we show $\{P_{i,k}\}$ is bounded in $\mathcal{S}^{n_i}_+\cap B_{\varepsilon}(P_i^*)$.
By picking  $\lambda_i$ sufficiently small, we know from part (i) of \cref{col1} that the couple system \cref{dmre_equ5} and \cref{dmre_equ6} can be made asymptotically stable at $(P_1^*,P_2^*)$, with $P_i(0)\in \mathcal{S}^{n_i}_+\cap B_{\varepsilon}(P_i^*)$;
and also from \cref{thm4} that  $\varepsilon$ can be made arbitrarily large.}

{\highlight  Now, choosing $\sup_k\{h_{i,k}\}$ sufficiently small, we easily have from part (i) of \cref{thm5} that $\{P_{i,k}\}$ stays in $\mathcal{S}^{n_i}_+\cap B_{\varepsilon}(P_i^*)$. 
Then, the proof is completed by part (iii) of \cref{thm5}.}
\end{proof}

\begin{remark}
The results presented in this section can be extended in different directions, such as for large-scale networks with more than two nodes and decentralized VI under stochastic disturbance.
\end{remark}



\section{Illustrative practical examples}\label{simulate}
In this section, we provide three simulation examples to illustrate our robust VI algorithm.

\subsection{Mean-variance portfolio optimization}\label{portfolio}
In this example, we study the mean-variance portfolio optimization problem \cite{Zhou2000} using non-zero-sum differential game theory and the robust VI results obtained in Sections \Cref{modelerr} and \Cref{decom}.

Consider the price process of $N+1$ assets (or securities) traded continuously in a market \cite{Zhou2000}:
\begin{align*}
\frac{dS_0}{S_0}&=rdt,\\
\frac{dS_i}{S_i}&=b_idt+\sum_{j=1}^{n_i} \sigma_{ij}dw_j,\quad i=1,2,\cdots,N,
\end{align*}
where $S_0$ represents the price of a bond, $S_i$, $i=1,\cdots,N$, represent $N$ stocks, $r>0$ is the interest rate,  $b_i>0$ is the appreciation rate, and $\{\sigma_{ij}\}_{j=1}^{n_i}$ is the volatility of the $i$-th stock.
An investor's total wealth at time $t$, when holding $h_i(t)$ shares of the $i$-th asset,  is given as 
\begin{align*}
x(t)=\sum_{i=0}^N h_i(t)S_i(t).
\end{align*}
Then, 
\begin{align*}
dx=\left(rx + \sum_i (b_i-r)u_i\right)dt+\sum_{i,j}\sigma_{ij}u_idw_j,
\end{align*}
where $u_i:=h_iS_i$ denotes the total market value of the investor's wealth in the $i$-th bond/stock.
The design objective here is to find $u$ to 
\begin{enumerate*}[label={\alph*)}]
\item maximize the average return; and
\item  minimize the volatility of $x$.
\end{enumerate*}

Inspired by  \cite{Zhou2000}, instead of solving the above portfolio optimization problem directly, we consider an auxiliary multi-player non-zero-sum differential game composed with the following cost
\begin{align}
J_i(u)=\mathbb{E}\left[\int_0^\infty \left(Q_i\bar x^2+\sum_{j=1}^NR_{ij}\bar u_i\bar u_j\right)dt\right],\quad i=1,\cdots N,\label{sim1_equ1}
\end{align}
subject to
\begin{align*}
d\bar x=\left(r\bar x + \sum_i (b_i-r)\bar u_i\right)dt+\sum_{i,j}\sigma_{ij}\bar u_idw_j,
\end{align*}
where $\bar x=x-\gamma$, and $\gamma>0$ represents the tradeoff between the two objectives in the portfolio optimization problem.
A larger $\gamma$ means more weights on the average return, and a small $\gamma$ means more weights on the volatility.
Note that the first term in the integrand in \cref{sim1_equ1} is related to the variance of $\bar x$ (and hence $x$) at the steady state, and the second term guarantee that the shares for the $i$-th bond/stock do not diverge to the infinity. 

Since the volatilities of assets are usually difficult to estimate, we borrow the idea of stochastic robust optimal solution from \cite[Section 5]{Bian2016b}, by choosing sufficiently small $Q_i>0$ and $R_{ij}>0$  to guarantee the small-gain condition. 
Then, the above non-zero-sum differential game can be solved using \cref{alg6}, with $A_i=r$ and $B_i=b_i-r$.
Based on the desired expected return, $\gamma$ is chosen as $200$.
Once $\bar u_i^*:=-K_i^*\bar x$ is obtained, the optimal share of the $i$-th asset at time $t$ is chosen as $K_i^*(\gamma-\bar x(t))$.
Totally $20$ stocks and one bond are used to construct the portfolio.
The interest rate is chosen as $2.5\%$, and the appreciation rates are randomly selected from $0-15\%$.
Suppose the real values of these rates are unknown, and are estimated online using techniques developed in \Cref{modelerr}.
After $1000$ iterations, all $P_i^*$ converge to their optimal values.
The prices of the portfolio  is shown in \cref{fig_Portfolio}.
Note that the portfolio constructed using the non-zero-sum differential game approach has a higher return, while maintaining approximately the same volatility compared with the uniform allocation of the asset.

\begin{figure}[htbp]
        \centering
\includegraphics[width=\textwidth]{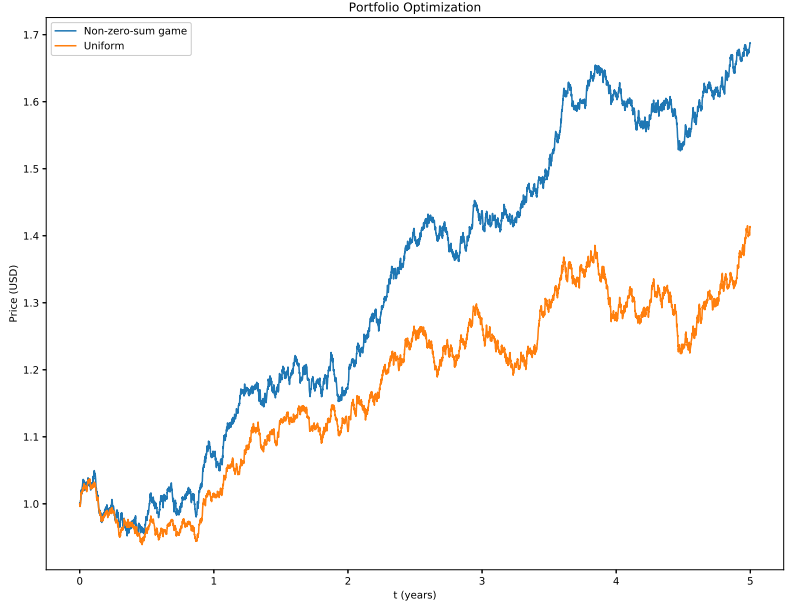}
\caption{Example \ref{portfolio}: Portfolio optimization solved from robust VI.}
\label{fig_Portfolio}
\end{figure}

\subsection{ADP learning for kinematic models}\label{kinematics}
In this example, we use the ADP method proposed in \Cref{VIADP} to develop an online learning mechanism for a class of kinematic models described as follows:
\begin{align*}
\dot p&=v,\\
m \dot v&=f-bv,\\
\tau \dot f&=u-f+w,
\end{align*}
where $p$, $v$, $f$ denote the  relative position to the origin, the velocity, and the actuator force, respectively; $u$ is the control input; 
$m$, $b$, and $\tau$ represent the mass, the viscosity constant, and the time constant, respectively;
and $w$ is an exploration noise used to facilitate the ADP learning.
Note that the above system can represent a large class of practical systems, including human motor system, autonomous vehicle model, power system, to name a few.


In practice, the state information collected from online data is usually corrupted by some observation noises.
As a result, instead of $(p,v,f)$, we assume only $(\hat p,\hat v,\hat f)$ is observed and used in the feedback control design and ADP learning:
\begin{align*}
\hat p&=p+\sigma_pw_p,\\
\hat v&=v+\sigma_vw_v,\\
\hat f&=f+\sigma_fw_f,
\end{align*}
where $\sigma_p$, $\sigma_v$, and $\sigma_f$ denote the noise magnitude; and 
$w_p$, $w_v$, and $w_f$ are independent random variable and follow standard Gaussian distribution.

\begin{table}
\centering
\caption{Parameters of the Kinematic Model.}
\begin{tabular}{ c c  l }
\cline{1-3}
Parameters &   Description & Value \\
\cline{1-3}
$m$                        &   Mass & $1$kg\\
$b$                        &  Viscosity constant & $1$N$\cdot$s$/$m \\
$\tau$                       &  Time constant & $0.1$s \\
$\sigma_p$                       &  Noise magnitude & $0.01$ \\
$\sigma_v$                       &  Noise magnitude & $0.02$ \\
$\sigma_f$                       &  Noise magnitude & $0.1$ \\
\cline{1-3}
\end{tabular}
\label{tab1}
\end{table}

The values of model parameters used in simulation are provided in  \cref{tab1}.
\cref{alg2} is applied online, and the control policy is updated in real time after every $0.02$s.
The weighting matrices in the cost are chosen as $Q=I_3$ and $R=I_1$.
The initial controller $u\equiv0$, i.e., only the exploration noise is injected into the system at the beginning.
The elements in $P_k$ are plotted in \cref{fig_App2P} for each $k$.
For comparison purpose, both the optimal solution $P^*$ and the near optimal solution $\hat P^*$ learned through ADP  are given below:
\begin{align*}
P^*=
\begin{bmatrix}
7.4044     & 1.4311  &0.1000\\
1.4311   & 0.3801  &0.0248\\
0.1000   & 0.0248  &0.0431\\
\end{bmatrix},\quad
\hat P^*=
\begin{bmatrix}
7.4560 & 1.5184 & 0.1084\\
1.5184 & 0.5117 & 0.0216\\
0.1084   & 0.0216  &0.0463\\
\end{bmatrix}.
\end{align*}
Obviously, $P^*$ and $\hat P^*$ are close to each other.
The system trajectories and input are given in \cref{fig_App2xu}.
Note that the system achieves the asymptotical stability in mean square sense.

\begin{figure}[htbp]
        \centering
\includegraphics[width=\textwidth]{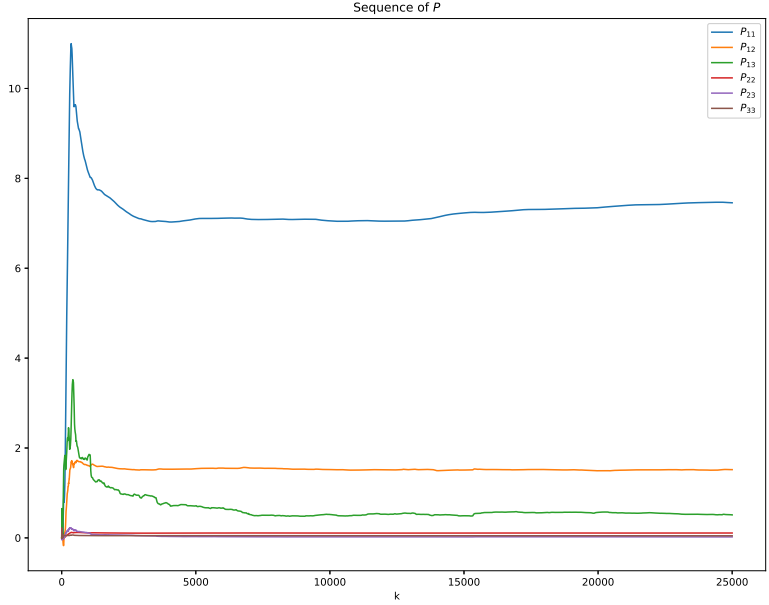}
\caption{Example \ref{kinematics}: elements of $P_k$.}
\label{fig_App2P}
\end{figure}

\begin{figure}[htbp]
        \centering
\includegraphics[width=\textwidth]{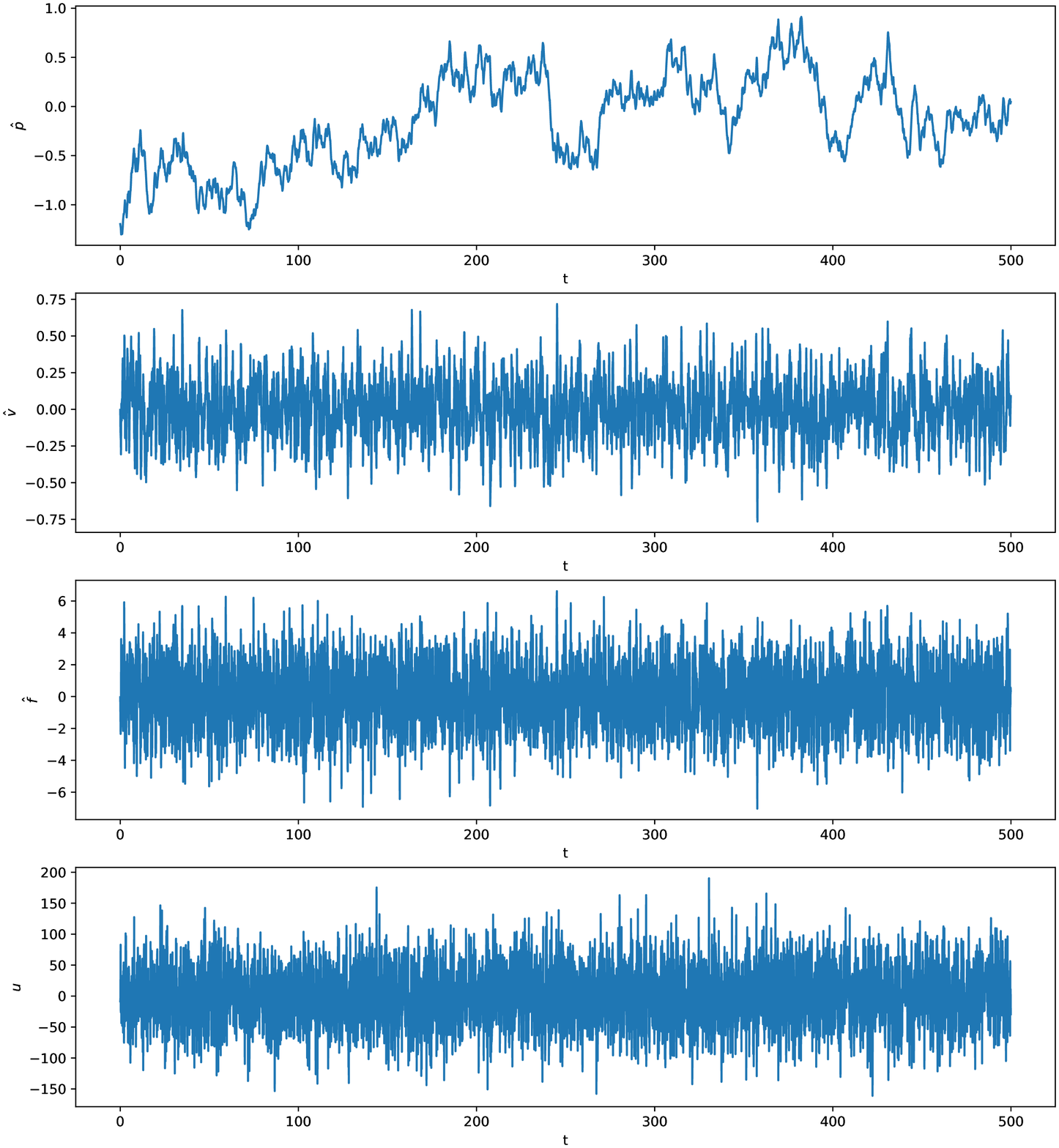}
\caption{Example \ref{kinematics}:  the trajectories of $\hat p$, $\hat v$, $\hat f$, and $u$.}
\label{fig_App2xu}
\end{figure}

%


\subsection{ADP for time-series variance minimization}\label{TSvar}
In this example, we use the ADP method developed in \Cref{stoADP} to study the variance minimization problem for a class of time-series with unknown parameters.
Note that this is a classical problem which has been studied in both finance and signal-processing community, and can be easily addressed using the Kalman filter, when the model parameters are known.

Consider the following time series in continuous-time:
\begin{align*}
\dddot S =\alpha_3 \ddot S + \alpha_2 \dot S + \alpha_1 S + \sigma_0v_0,
\end{align*}
where $ \sigma_0$ and $\alpha_i$, $i=1,2,3$, are unknown model parameters; and
$v_0$ is a Gaussian white noise that drives the output $S$.
Suppose the system is asymptotically stable in mean square sense.
Our objective here is to minimize the variance of $S$.

By rewriting the above differential equation in state space form, we have
\begin{align*}
d x_1 &= x_2 dt + \sigma_1d w_1 - \sigma_2w_2 dt,\\
d x_2 &= x_3 dt+ \sigma_2d w_2 - \sigma_3w_3 dt,\\
d x_3 &= \alpha_3 x_3 dt+  \alpha_2 x_2 dt+ \alpha_1 x_1 dt+  \sigma_3d w_3+ u +  \sigma_0d w_0,
\end{align*}
where $x_1=S+ \sigma_1w_1$, $x_2=\dot S+ \sigma_2w_2$, $x_3=\ddot S + \sigma_3w_3$;
$w_0=\dot v_0$;
 $w_i$, $i=0, 1,2,3$, are  Brownian motions representing the observation noises;
 $\sigma_i$, $i=1,2,3$, are unknown noise magnitudes;
  and $u$ is the control input.
  Note that even if $u\equiv0$, $\mathbb{E}x_i$, $i=1,2,3$, can decrease to $0$ asymptotically, since we assume the system is asymptotically stable in mean square sense.
  However, the variance of $x_i$ may be extremely large due to the presence of $\sigma_0v_0$.
  To reduce the variance of $x_i$, \cref{alg5} is used to develop an ergodic controller.
  Notice that by the law of large numbers, $\int_0^\infty w_i dt=0$ for all $i$.
  Hence, the two terms $\sigma_2w_2 dt$ and $\sigma_3w_3 dt$ have little influence in the time integration in \cref{alg5}.

In the simulation, we choose $\alpha_1=-4$, $\alpha_2=-1$, $\alpha_3=-4$, $\sigma_0=1$, $\sigma_1=0.6$, $\sigma_2=0.4$, and $\sigma_3=0.5$.
For illustration purpose, the weighting matrices in the cost are chosen as $Q=0.1I_3$ and $R=0.01I_1$.
$P_k$ is updated in real time after every $1$s.
The elements in $P_k$ are given in \cref{fig_App3P}.
Both the optimal solution $P^*$ and the near optimal solution  $\hat P^*$ from ADP learning are shown below:
\begin{align*}
P^*=
\begin{bmatrix}
0.2859  & 0.1492  & 0.0110\\
0.1492  & 0.3366  & 0.0539\\
0.0110  & 0.0539  & 0.0206\\
\end{bmatrix},\quad
\hat P^*=
\begin{bmatrix}
0.2854 &  0.1479 &  0.0106\\
0.1479 &  0.3377 &  0.0529\\
0.0106 &  0.0529 &  0.0262\\
\end{bmatrix}.
\end{align*}
The system trajectories are given in \cref{fig_App3xu}.
Note that the controller derived from \cref{alg5} significantly reduces the variance of the output signal.
\begin{figure}[htbp]
        \centering
\includegraphics[width=\textwidth]{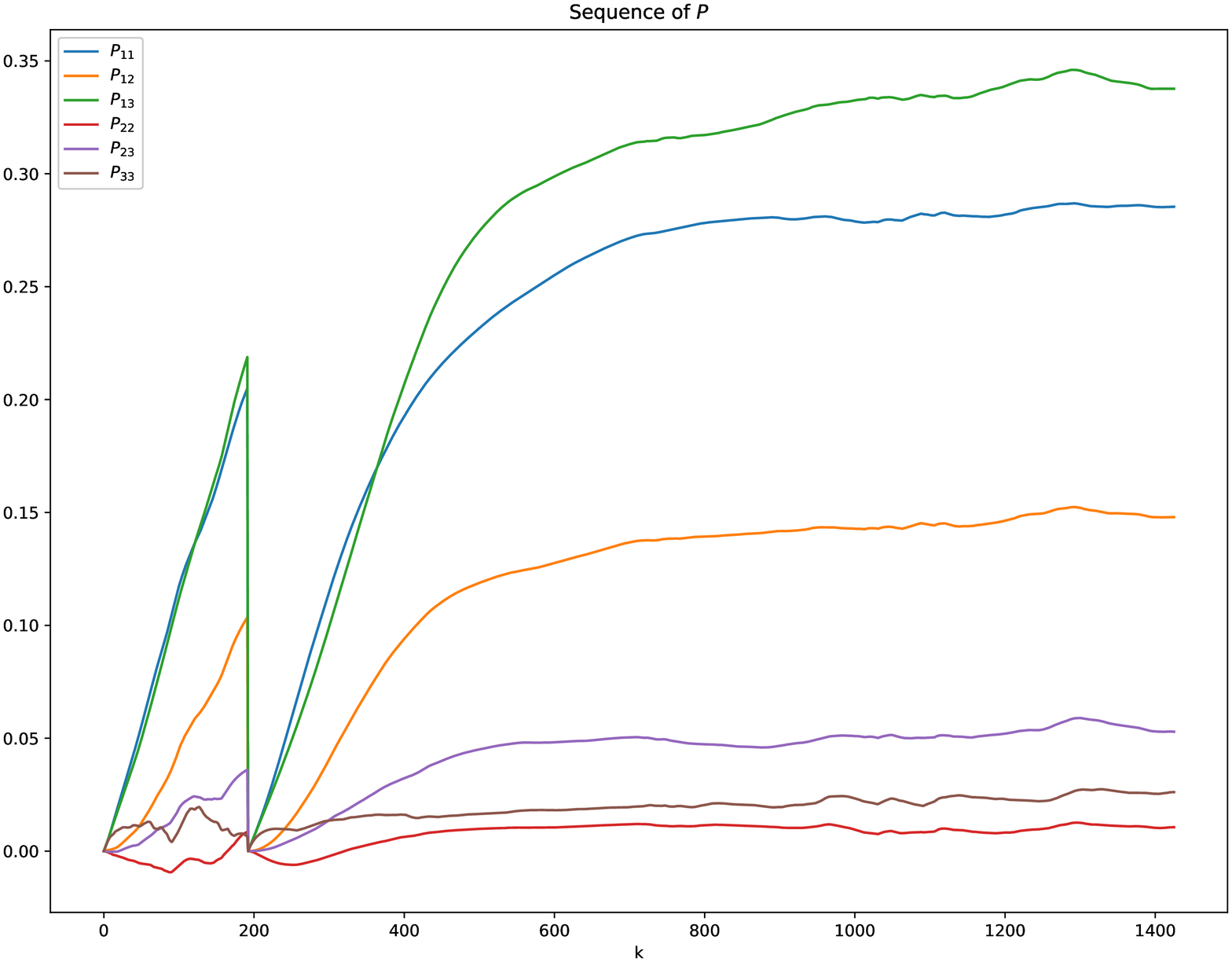}
\caption{Example \ref{TSvar}:  elements of $P_k$.}
\label{fig_App3P}
\end{figure}

\begin{figure}[htbp]
        \centering
\includegraphics[width=\textwidth]{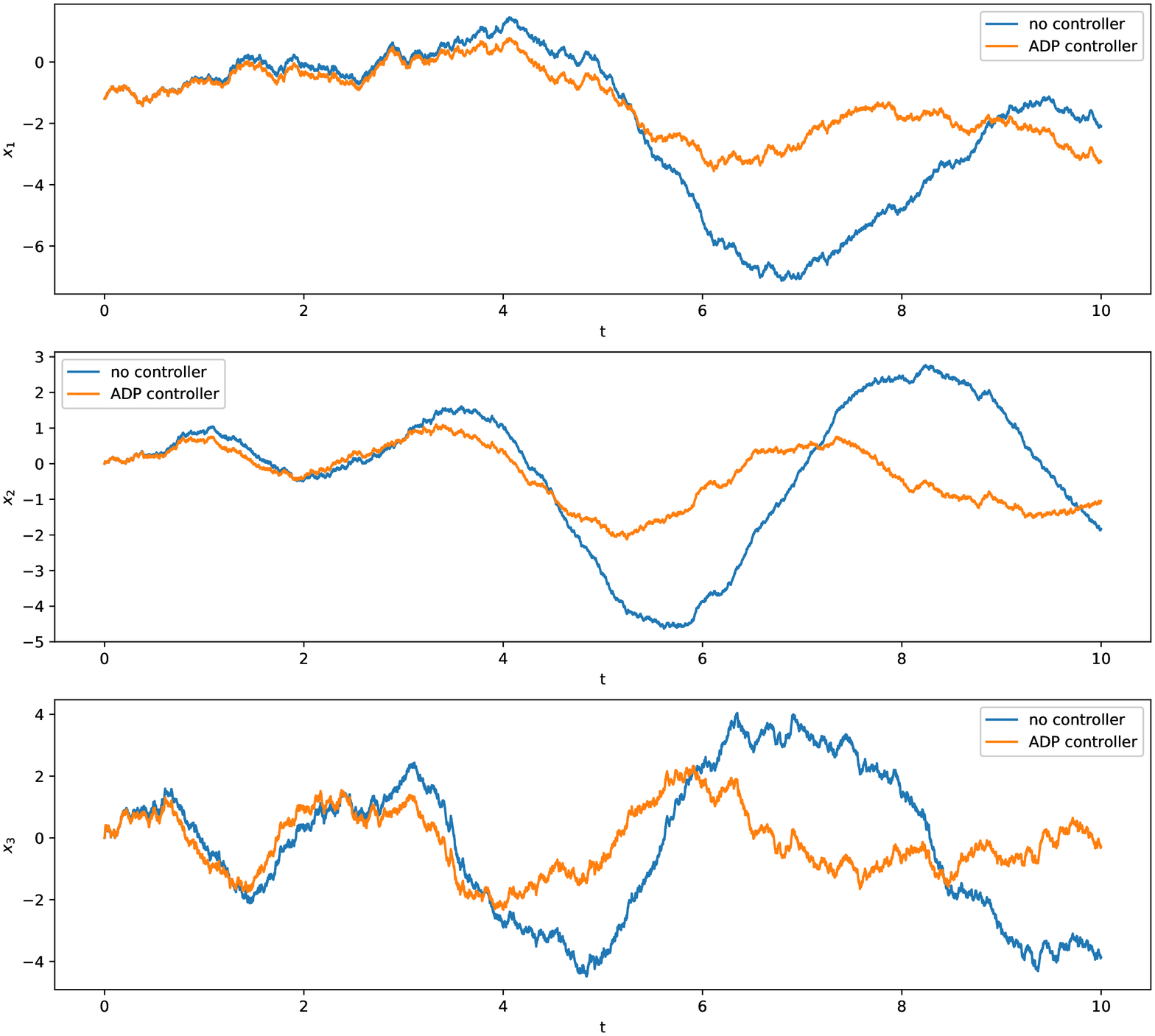}
\caption{Example \ref{TSvar}:  the trajectories of $x_i$, $i=1,2,3$.}
\label{fig_App3xu}
\end{figure}

\section{Summary and future work}

This paper develops a new framework of robust DP.
This novel theory resolves a long-standing issue in DP theory: how to develop DP algorithms that are robust to different types of disturbances?
Empowered by nonlinear and robust control theories, robust DP allows us to develop various DP and RL algorithms with guaranteed convergence to the optimal solution in the presence of different types of disturbances, including  stochastic noise,  external disturbances, and modeling errors such as  nonlinear dynamic uncertainties.
To be specific, we have conducted an innovative input-output gain analysis for the DMRE in \Cref{RVI}, and applied the  result together with the nonlinear small-gain theory to develop a novel robust VI algorithm.
It has been shown that this new algorithm is robust to different kinds of internal and external disturbances, and hence is especially useful in solving non-model-based optimal control problems.

Due to space limitations, we only list a few illustrative applications of our robust DP method in \Cref{sec_applicaiton}.
These examples have demonstrated that robust DP obtained in the present paper is a powerful tool for addressing adaptive optimal control and DP problems.
Last but not least, we point out several additional research topics that deserve further investigations in the future  on the basis of robust DP:
\begin{itemize}
\item {\it Robust PI.}
This paper mainly focuses on developing the robust VI.
Due to the popularity of the PI algorithm in real-world decision making applications, it is important to develop similar robustness analysis results for the PI.
\item {\it Multi-level ADP learning.}
The convergence of ADP algorithms relies heavily on the well-posedness of the cost functional.
However, it may not be easy to identify such a ``qualified'' cost in practice.
One way of solving this problem is to update the cost functional at the same time when the ADP learning is performed.
The convergence of this multi-level learning algorithm can be analyzed using our robust DP framework.
\item {\it Neural network-based ADP methods.}
Robust DP can also play an important role in analyzing the convergence of nonlinear ADP methods with neural network approximation.
The  error induced from neural network approximation can be regarded as an external input to the robust DP algorithm.
Then the gain analysis can be conducted to quantify the influence of such approximation error.
\item {\it Robust ADP learning under unknown disturbance.}
Robust ADP aims at developing a robust adaptive optimal controller for an interconnected system subject to dynamic uncertainty.
A potential drawback of previous robust ADP algorithms is that the disturbance input must be accessible during the learning process. 
This restrictive assumption can be removed with the help of the proposed robust DP methodology.
\item {\it DP and ADP methods for delayed systems.}
The time delay can be handled as a special type of dynamic uncertainty with the unity gain. 
Since robust DP provides a new way of conducting the convergence analysis from a nonlinear small-gain perspective, it can play a vital role in handling the adaptive optimal control design with delayed input and state information.
\item {\it Parallel  and decentralized ADP methods.}
\Cref{decom} has developed a decentralized VI algorithm that is shown especially useful in solving large-scale DP problems and differential games.
Our future work will be directed at developing the model-free counterpart of \cref{alg6}, and extending this result to more general scenarios.
\end{itemize}

\section*{Acknowledgements}
The work of Z.P. Jiang has been supported partially by the National Science Foundation under Grants ECCS-1230040 and ECCS-1501044.

\bibliographystyle{alpha}
\bibliography{VIbib.bib}
\end{document}